\documentclass[11pt]{article}
\usepackage[top=2cm,bottom=2.5cm,right=2.5cm,left=2.5cm]{geometry}
\usepackage[english]{babel}
\usepackage[utf8]{inputenc}
\usepackage{times}
\usepackage{color,pdfpages}
\usepackage{mathtools, bm}
\usepackage{amssymb, bm}
\usepackage{graphicx}
\usepackage{mathrsfs}
\usepackage{stmaryrd}
\usepackage{amsthm}
\usepackage{listings}
\usepackage{pict2e}
\usepackage{pgf, tikz}
\usepackage{dsfont}
\usepackage{algorithm2e}
\usepackage{algorithmic}
\usepackage{multicol}
\usepackage{hyperref}

\thicklines

\newtheorem{theorem}{Theorem}
\newtheorem{proposition}{Assumption}

\newtheorem{lemma}{Lemma}
\newtheorem{remark}{Remark}

\newtheorem{example}{Example}

\numberwithin{equation}{section} 
\numberwithin{theorem}{section} 
\numberwithin{lemma}{section}
\numberwithin{remark}{section}
\numberwithin{corollary}{section}
\numberwithin{example}{section}
\numberwithin{definition}{section}

\allowdisplaybreaks

\newcommand{\po}{\left(}
\newcommand{\pf}{\right)}

\newcommand{\cco}{\llbracket}
\newcommand{\ccf}{\rrbracket}
\newcommand{\R}{\mathbb{R}}

\date{}

\title{Convergence rates for the Vlasov-Fokker-Planck equation and uniform in 
time propagation of chaos in non convex cases}
\author{Arnaud Guillin\footnote{Laboratoire de Math\'{e}matiques Blaise Pascal - Universit\'{e} Clermont-Auvergne. Email : arnaud.guillin[AT]uca.fr},$\ $ Pierre Le Bris\footnote{ Laboratoire Jacques-Louis Lions - Sorbonne Université. Email : pierre.lebris[AT]sorbonne-universite.fr}$\ $ and Pierre Monmarch\'{e}\footnote{ Laboratoire Jacques-Louis Lions - Sorbonne Université. Email : pierre.monmarche[AT]sorbonne-universite.fr}}

\begin{document}
\maketitle

\begin{abstract}
We prove the existence of a contraction rate for Vlasov-Fokker-Planck equation in Wasserstein distance, provided the interaction potential is (locally) Lipschitz continuous and the confining potential is both Lipschitz continuous and greater than a quadratic function, thus requiring no convexity conditions. Our strategy relies on coupling methods suggested by  A. Eberle \cite{Eberle_reflection} adapted to the kinetic setting enabling also to obtain uniform in time propagation of chaos in a non convex setting.
\end{abstract}

%\tableofcontents

\section{Introduction}

\subsection{Framework}

Let $U$ and $W$ be two functions in $\mathcal{C}^1\left(\mathbb{R}^d\right)$. We consider the Vlasov-Fokker-Planck equation: 
\begin{equation}\label{Liouville}
\partial_t\nu_t\left(x,v\right)=-\nabla_x\cdot\left(v\nu_t\left(x,v\right)\right)+\nabla_v\cdot\left(\left(v+\nabla U\left(x\right)+\nabla W\ast\mu_t\left(x\right)\right)\nu_t\left(x,v\right)+\nabla_v\nu_t\left(x,v\right)\right),
\end{equation}
where $\nu_t(x,v)$ is a probability density in the space of positions $x\in\mathbb{R}^d$ and velocities $v\in\mathbb{R}^d$,
\begin{equation*}
\mu_t\left(x\right)=\int_{\mathbb{R}^d}\nu_t\left(x,dv\right)
\end{equation*}
is  the space marginal of $\nu_t$ and
\begin{equation*}
\nabla W\ast\mu_t(x)=\int_{\mathbb{R}^d}\nabla W(x-y)\mu_t(dy).
\end{equation*}
It has the following probabilistic counterpart, the non linear stochastic 
differential equation of \textit{McKean-Vlasov} type, i.e. $\nu_t$ is the 
density of the law at time $t$ of the $\mathbb{R}^{2d}$-valued process $(X_t,V_t)_{t\ge 0}$ evolving as the mean field SDE (diffusive Newton's equations)
\begin{equation}\label{FP}
\left\{
    \begin{array}{ll}
        dX_t=V_tdt\\
        dV_t=\sqrt{2}dB_t-V_tdt-\nabla U\left(X_t\right)dt-\nabla W\ast\mu_t\left(X_t\right)dt\\
        \mu_t=\text{Law}\left(X_t\right).
    \end{array}
\right.
\end{equation}
Here, $\left(X_t,V_t\right)\in\mathbb{R}^d\times\mathbb{R}^d$, $\left(B_t\right)_{t\geq0}$ is a Brownian motion in dimension $d$ on a probability space $\left(\Omega, \mathcal{A},\mathbb{P}\right)$, and $\mu_t$ is the law of the position $X_t$. The symbol $\nabla$ refers to the gradient operator, and the symbol $\ast$ to the operation of convolution. 

Both in the probability and in the partial differential equation community, existence and uniqueness of McKean-Vlasov processes have been well studied. See \cite{McKean}, \cite{Funaki}, \cite{Saint_Flour_1991} for some historical milestones. In the specific case of \eqref{Liouville} and \eqref{FP}, under the assumptions on $U$ and $W$ introduced in the next section, existence and uniqueness follow from \cite{Meleard_Asymptotic}   for square integrable initial data.

A related process is the $N$ particles system in $\mathbb{R}^d$ in mean field interaction
\begin{equation}\label{FP_part}
\forall i\in \cco 1,N\ccf\,,\qquad \left\{
    \begin{array}{rcl}
        dX^i_t &=& V^i_tdt,\\
        dV^i_t &= & \displaystyle{\sqrt{2}dB^i_t-V^i_tdt-\nabla U\left(X^i_t\right)dt-\frac{1}{N}\sum_{j=1}^N\nabla W\left(X^i_t-X^j_t\right)dt},
        \end{array}
\right.
\end{equation}
where $X^i_t$ and $V^i_t$ are respectively the position and the velocity of the i-th particle, and $\left(B^i_t, 1\leq i\leq N\right)$ are independent Brownian motions in dimension $d$. One can see equation~\eqref{FP_part} as an approximation of equation~\eqref{FP}, where the law $\mu_t$ is replaced by the empirical measure $\mu_t^N=\frac{1}{N}\sum_{i=1}^N\delta_{X^i_t}$. 

It is well known, at least in a non kinetic setting (\cite{Meleard_Asymptotic}, \cite{Saint_Flour_1991}), that, under some weak conditions on $U$ and $W$, $\mu_t^N$ converges in some sense toward the law $\mu_t$ of $X_t$ solution of \eqref{FP}. This phenomenon has been stated under the name \textit{propagation of chaos}, an idea motivated by M. Kac \cite{kac1956}), and greatly 
developed by A.S. Sznitman \cite{Saint_Flour_1991}.

In statistical physics, \eqref{FP_part} is a Langevin equation that describes the motion of $N$ particles subject to damping, random collisions and a \emph{confining potential} $U$ and interacting with one another through an \emph{interaction potential} $W$, which can be polynomial (granular 
media), Newtonian (interacting stellar) or Coulombian (charged matter). See for instance \cite{Langevin_anglais} for an english translation of P. Langevin's landmark paper on the physics behind the standard underdamped Langevin dynamics. Therefore, Equation~\eqref{Liouville}  has the following natural interpretation: the solution $\nu_t$ is the density of the law 
at time $t$ of the process $(X_t,V_t)_{t\geq0}$ evolving according to \eqref{FP}, and thus describes the limit dynamic of a cloud of (charged) particles. In particular, it holds importance in plasma physics, see \cite{Vlasov_1968}. 

More recently, mean-field processes such as \eqref{FP_part} have drawn much interest in the analysis of neuron networks in machine learning \cite{cheng2018underdamped,cheng2018sharp}. In this context of stochastic algorithms, it is known that the underdamped Langevin dynamics (not necessarily with mean-field interactions) can converge faster than the overdamped (i.e non kinetic) Langevin dynamics \cite{cheng2018underdamped,GM16} toward its invariant measure. For example, the results on \eqref{FP} could be applied to study the convergence of the Hamiltonian gradient descent algorithm for the overparametrized optimization as done in \cite{kazeykina2020ergodicity} for Generative Adversarial Network training.

The goal of the present work is twofold. We are interested, first,  in the long-time convergence of the solution of \eqref{FP} toward an equilibrium and, second, to a uniform in time convergence as $N\rightarrow +\infty$ of   \eqref{FP_part} toward   \eqref{FP}. It is well known that such results cannot hold in full generality, as the non-linear equation \eqref{Liouville} may have several equilibria. Here we will consider cases where the interaction is sufficiently small for the non-linear equilibrium to be unique and globally attractive, and for the propagation of chaos to be uniform in time.

There are various methods to study the long time behavior of kinetic type 
processes, such as Lyapunov conditions or hypocoercivity, and we will discuss these approaches and compare them with our results later on.
We rely here on coupling methods following the guidelines of A. Eberle \textit{et al.} in \cite{Eberle_Guillin_Zimmer_Langevin} where  the convergence to equilibrium is established for \eqref{FP} without interaction, and also extend the approach to handle only locally Lipschitz coefficient. In a second part, we also use reflection couplings (see \cite{Uniform_Prop_Chaos}) for the propagation of chaos property. 

Let us briefly describe the coupling method. The basic idea is that an upper bound on the Wasserstein distance between two probability distributions is given by the construction of any pair of  random variables distributed respectively according to those. The goal is thus to construct simultaneously two solutions of \eqref{FP} that have a trend to get closer with 
time. Have $\left(X_t,V_t\right)$ be a solution of \eqref{FP} driven by some Brownian motion $(B_t)_{t\geqslant 0}$ and let $\left(X'_t, V'_t\right)$ solves
\begin{equation*}
\left\{
    \begin{array}{ll}
        dX_t'=V_t'dt\\
        dV_t'=\sqrt{2}dB_t'-V_t'dt-\nabla U\left(X_t'\right)dt-\nabla W\ast\mu_t\left(X_t'\right)dt\\
        \mu_t'=\text{Law}\left(X_t'\right).
    \end{array}
\right.
\end{equation*}
with $\left(B'_t\right)_{t\geqslant 0}$ a  d-dimensional Brownian motion. 
 A coupling of $(X,V)$ and $(X',V')$ then follows from a coupling of the Brownian motions $B$ and $B'$. Choosing $B=B'$ yields the so-called \textit{synchronous} coupling, for which the Brownian noise cancels out in the infinitesimal evolution of the  difference $\left(Z_t,W_t\right)=\left(X_t-X'_t,V_t-V'_t\right)$. In that case the contraction of a distance between the processes can only be induced by the deterministic drift, as in \cite{BGM10}. Such a deterministic contraction only holds under very restrictive conditions, in particular $U$ should be strongly convex. Nevertheless, in more general cases, the calculation of the evolution of $Z_t$ 
and $W_t$ (see Section~\ref{sec:step_one} below) shows that there is still some deterministic contraction when $Z_t+W_t=0$. We can therefore  use a synchronous coupling in the vicinity of this subspace.

Outside of $\{(z,w)\in\mathbb R^{2d}, z+w=0\}$, it is necessary to make 
use of the noise to get the processes closer together, at least in the direction orthogonal to this space. In order to maximize the variance of this noise, we then use a so-called  \textit{reflection} coupling, which consists in $B$ and $B'$ being \textit{symmetrical}  (i.e $B_t'=-B_t$) in 
the direction of space given by the difference of the processes, and synchronous in the orthogonal direction. In other words, writing
\begin{align*}
e_t=\left\{
    \begin{array}{ll}
    \frac{Z_t+W_t}{|Z_t+W_t|}&\text{ if }Z_t+W_t\neq0\\
    0&\text{ otherwise}
    \end{array}
\right.
\end{align*}
we consider $dB'_t=\left(Id-2e_te^T_t\right)dB_t$. Levy's characterization then ensures that it is indeed a Brownian motion. 

Finally we construct a Lyapunov function $H$ to take into account the trend of each process to come back to some compact set of $\mathbb R^{2d}$. We are then led to the study of a suitable distance between the two processes, which will be of the form $\rho_t:=f(r_t)(1+\epsilon H(X_t,V_t)+\epsilon H(X'_t,V'_t))$, with $r_t=\alpha|Z_t|+|Z_t+W_t|$, where $\alpha,\epsilon>0$ and the function $f$ are some parameters to choose. More precisely, we have to choose these parameters carefully in order for $\mathbb{E}\rho_t$ to decay exponentially fast. This leads to several constraints on $\alpha,\epsilon$ and on the parameters involved in the definition of $f$, and we have to prove that it is possible to meet all these conditions simultaneously. For the sake of clarity, in fact, we present the proof in a different order, namely we start by introducing very specific parameters and, throughout the proof, we check that our choice of parameters implies the needed constraints.

The study of the limit $N\rightarrow +\infty$ is based on a similar coupling, except that we couple a system of $N$ interacting particles \ref{FP_part} with $N$ independent non-linear processes \ref{FP}.

The next subsections describe our main results and compare them to the few existing ones in the literature. Section~\ref{section_2} presents the precise construction of the aforementioned {\it ad hoc} Wasserstein distance. The proof of the long time behavior of the Vlasov-Fokker-Planck equation when confinement and interaction coefficient are Lipschitz continuous is done in Section~\ref{Preuve_1}, 
whereas the propagation of chaos property is proved in Section~\ref{Preuve_2}. An appendix gathers technical lemmas and the modifications of the main proofs when the confinement is only supposed locally Lipschitz continuous.

\subsection{Main results}

For $\mu$ and $\nu$  two probability measures on $\mathbb{R}^{2d}$, denote by $\Pi\po\mu,\nu\pf$ the set of couplings of $\mu$ and $\nu$, i.e. the 
set of probability measures $\Gamma$ on $\mathbb{R}^{2d}\times\mathbb{R}^{2d}$ with $\Gamma(A\times \R^{2d}) = \mu(A)$ and $\Gamma(\R^{2d}\times 
A ) = \nu(A)$ for all Borel set $A$ of $\R^{2d}$. We will define $L^1$ and $L^2$ Wasserstein distances as  
\begin{align*}
\mathcal{W}_1\left(\mu,\nu\right)&=\inf_{\Gamma\in\Pi\left(\mu,\nu\right)}\int\left(|x-\tilde{x}|+|v-\tilde{v}|\right)\Gamma\left(dxdvd \tilde{x}d\tilde{v} \right),\\
\mathcal{W}_{2}\left(\mu,\nu\right)&=\left(\inf_{\Gamma\in\Pi\left(\mu,\nu\right)}\int\left(|x-\tilde{x}|^2+|v-\tilde{v}|^2\right)\Gamma\left(dxdvd \tilde{x}d\tilde{v}  \right)\right)^{1/2}\,.
\end{align*}
Our main results will be stated in terms of these distances, even if we work and get contraction in the Wasserstein distance defined with the aformentioned $\rho$. Let us detail the assumptions on the potentials $U$ and $W$.
\begin{proposition}\label{HypU}
The potential $U$ is non-negative and there exist $\lambda>0$ and $A\geq 0$ such that 
\begin{equation}\label{HypU-lya}
\forall x\in\mathbb{R}^d\,,\qquad \frac{1}{2}\nabla U\left(x\right)\cdot x\geq\lambda\left(U\left(x\right)+\frac{|x|^2}{4}\right)-A.
\end{equation}
\end{proposition}

 The condition \eqref{HypU-lya} implies that the force $-\nabla U$ has a confining effect, bringing back particles toward some compact set.   It implies the following:
 \begin{lemma}\label{minU}
Provided \eqref{HypU-lya}, there exists $\tilde A\geq 0$ such that for all $x\in\R^d$,
\begin{equation}\label{tilde_A}
 U\left(x\right)\geq\frac{\lambda}{6}|x|^2-\tilde{A}.
\end{equation}
\end{lemma}
The proof is postponed to Appendix~\ref{preuve_min_U}.  In particular,  it implies that $U$ goes to infinity at infinity and is bounded below.  Since only the gradient of $U$ is involved in the dynamics, the condition $U\geqslant 0$ is thus not restrictive as it can be enforced without loss of generality by adding a sufficient large constant to $U$. This condition is added  in order to simplify some calculations.

\begin{proposition}\label{HypU_lip}
There is a constant $L_U>0$ such that 
\begin{equation*}
\forall x,y\in\mathbb{R}^d\times\mathbb{R}^d\,,\qquad |\nabla U\left(x\right)-\nabla U\left(y\right)|\leq L_U|x-y|.
\end{equation*}
\end{proposition}

\begin{example}\label{double_puits}
The  double-well potential given by 
\begin{equation*}
U\left(x\right)=\left\{
    \begin{array}{ll}
        \left(x^2-1\right)^2&\text{ if }|x|\leq1,\\
        \left(|x|-1\right)^2&\text{ otherwise. }
    \end{array}
\right.
\end{equation*}
satisfies Assumptions~\ref{HypU} and \ref{HypU_lip}.
\end{example}

%Hypothèse

\begin{proposition}\label{HypW}
The potential $W$ is  even, i.e. $W\left(x\right)=W\left(-x\right)$ for 
all $x\in\R^d$, in particular $\nabla W\left(0\right)=0$. Moreover, there exists $ L_W<\lambda/8$ (where $\lambda$ is given in Assumption~\ref{HypU}) such that
\begin{equation}\label{HypW-lip}
\forall x,y\in\mathbb{R}^d\times\mathbb{R}^d,\qquad |\nabla W\left(x\right)-\nabla W\left(y\right)|\leq L_W |x-y|.
\end{equation}
In particular $|\nabla W\left(x\right)|\leq L_W |x|$ for all $x\in\R^d$.
\end{proposition}

Here we consider an interaction force that is the gradient of a potential 
$W$, as we stick to the formalism of other related works (for instance \cite{Uniform_Prop_Chaos}). Nevertheless, all the results and proofs still holds if $\nabla W$ is replaced by some $F:\R^d\mapsto \R^d$ satisfying the same conditions. The confinement potential may also be non gradient, however the fact that the confinement force $\nabla U$ is a gradient simplifies the construction of a Lyapunov function.

The condition $L_W\leq\lambda/8$ is related to the fact the interaction is considered as a perturbation of the non-interacting process studied in \cite{Eberle_Guillin_Zimmer_Langevin}. Therefore, $\nabla W$ has to be controlled by $\nabla U$ in some sense. Note that we immediately get the following bound  on the non-linear drift:

%Lemma

\begin{lemma}\label{majNablaW}
Under Assumption~\ref{HypW}, for all probability measures $\mu$ and $\nu$ 
on $\mathbb{R}^d$ and $x,\tilde{x}\in \mathbb{R}^d$,
\begin{equation*}
 |\nabla W\ast\mu\left(x\right)-\nabla W\ast\nu\left(\tilde{x}\right)|\leq L_W|x-\tilde{x}|+L_W\mathcal{W}_1(\mu,\nu).
\end{equation*}
\end{lemma}

See Appendix~\ref{preuve_majNablaW} for the proof.

\begin{example}
Assumption~\ref{HypW} is satisfied for an harmonic interaction $W(x) = \pm L_W|x|^2/2$, or
a  mollified Coulomb interaction for $a,b>0$ and $k\in\mathbb{N}^*$
\begin{equation*}
W\left(x\right)=\pm \frac{a}{\left(|x|^k+b^k\right)^{\frac{1}{k}}},
\ \ \ \text{ i.e }\ \ \ 
\nabla W\left(x\right)=\mp \frac{a x|x|^{k-2}}{\left(|x|^k+b^k\right)^{1+\frac{1}{k}}}.
\end{equation*}
\end{example}

We are also interested in cases where, instead of Assumption~\ref{HypU_lip}, $\nabla U$ is only assumed to be locally Lipschitz continuous.

\begin{proposition}\label{HypU_loc_lip}
There exist  $L_U>0$ and a function a function $\psi : \mathbb{R}^d\mapsto\mathbb{R}$    such that
\begin{equation*}
\forall x,y\in\mathbb{R}^d\times\mathbb{R}^d,\qquad \ |\nabla U\left(x\right)-\nabla U\left(y\right)|\leq\left(L_U+\psi\left(x\right)+\psi\left(y\right)\right)|x-y|,
\end{equation*}
and
\begin{equation*}
\forall x\in\mathbb{R}^d,\qquad 0\leq\psi(x)\leq L_\psi \sqrt{\lambda|x|^2+24U(x)},
\end{equation*}
where  $L_\psi>0$ is sufficiently small in the sense that
\begin{align*}
L_\psi\leq c_\psi(L_U, \lambda,\tilde{A},d,a),
\end{align*}
where $c_\psi$ is an explicit function given below in \eqref{c_psi}, $L_U$ is given in Assumption~\ref{HypU_lip}, $\lambda$ by Assumption~\ref{HypU}, $\tilde{A}$ by Assumption~\ref{HypU}  Lemma~\ref{minU} and $d$ is the dimension. Finally, $a$ is a used to bound an initial moment (see \eqref{init_cdt_loc_lip}).
\end{proposition}

The first of our main results concern the long-time convergence of the non-linear system \eqref{Liouville}. 

%Theoreme

\begin{theorem}\label{thm_1}
Let $U\in\mathcal{C}^1\left(\mathbb{R}^{d}\right)$ satisfy Assumption~\ref{HypU} and either Assumption~\ref{HypU_lip} or Assumption~\ref{HypU_loc_lip}). There is an explicit   $c^W>0$ such that, for all $W\in\mathcal{C}^1\left(\mathbb{R}^{d}\right)$ satisfying Assumption~\ref{HypW} with $L_W<c^W$, there is an explicit   $\tau>0$  such that for all probability measures $\nu^1_0$ and $\nu^2_0$ on $\mathbb{R}^{2d}$ with either a finite second moment (if Assumption~\ref{HypU_lip} holds) or a finite Gaussian moment (if only Assumption~\ref{HypU_loc_lip} holds), there are explicit constants  $C_1 ,C_2>0$ such that for all $t\geqslant 0$,
\[  \mathcal{W}_{1}\left(\nu^1_t,\nu^2_t\right) \leq e^{-\tau t}C_1\,,\qquad
 \mathcal{W}_{2}\left(\nu^1_t,\nu^2_t \right)  \leq e^{-\tau t}C_2\]
where $\nu^1_t$ and $\nu^2_t$ are solutions of \eqref{Liouville} with respective initial distributions $\nu^1_0$ and $\nu^2_0$.

In particular, we have existence and unicity of -as well as convergence towards - a stationary solution.

\end{theorem}

The second of our main results is a uniform in time convergence as $N\rightarrow +\infty$ of \eqref{FP_part} toward \eqref{FP}. 

%Theoreme
\begin{theorem}\label{thm_2}
Let $\tilde{\mathcal{C}}^0>0$ and $\tilde{a}>0$. Let $U\in\mathcal{C}^1\left(\mathbb{R}^{d}\right)$ 
satisfy Assumptions~\ref{HypU} and  \ref{HypU_lip}. There is an explicit  
$c^W>0$ such that, for all $W\in\mathcal{C}^1\left(\mathbb{R}^{d}\right)$ 
satisfying Assumption~\ref{HypW} with $L_W<c^W$, there exist explicit ${B_1,B_2>0}$,  such that for all probability measures $\nu_0$  on $\mathbb{R}^{2d}$ satisfying ${\mathbb{E}_{\nu_0}\left(e^{\tilde{a}\left(|X|+|V|\right)}\right)\leq\tilde{\mathcal{C}^0}}$ , 
\[
 \mathcal{W}_{1}\left(\nu^{k,N}_t,\bar{\nu}_t^{\otimes k}\right)\leq  \frac{k B_1}{\sqrt{N}}  \,,\qquad 
 \mathcal{W}_{2}^2\left(\nu^{k,N}_t,\bar{\nu}_t^{\otimes k}\right)\leq \frac{k B_2}{\sqrt{N}} ,
\]
for all $k\in\mathbb{N}$, where $\nu^{k,N}_t$  is the marginal distribution at time $t$ of the first $k$ particles $\left((X^1_t,V^1_t),....,(X^k_t,V^k_t)\right)$ of an $N$ particle system \eqref{FP_part}  with initial distribution $(\nu_0)^{\otimes N}$, while $\bar \nu_t$ is a solution of \eqref{Liouville} with  initial distribution $\nu_0$.
\end{theorem}

\subsection{Comparison to existing works}

Space homogeneous models of diffusive and interacting granular media (see 
\cite{BCCP}) have attracted a lot of attention, usually named McKean-Vlasov diffusions, these past twenty years, by means of a stochastic interpretation and synchronous couplings as in \cite{CGM08} or the recent \cite{Uniform_Prop_Chaos} by reflection couplings enabling to get rid of convexity conditions, but limited to small interactions. Remark however that small interactions are natural to get uniform in time propagation of chaos as 
for large interactions the non linear limit equation may have several stationary measures (see \cite{HT} for example). The granular media equations were interpreted as gradient flows in the space of probability measures in \cite{CMV03}, leading to explicit exponential (or algebraic for non uniformly convex cases) rates of convergence to equilibrium of the non linear equation. Another approach relying on the dissipation of the Wasserstein distance and $WJ$ inequalities was introduced in \cite{BGG13} handling small non convex cases. This approach was implemented in \cite{S20} to 
get propagation of chaos, under roughly the same type of assumptions.

Results on the long time behavior of the non-linear equation \eqref{FP}, i.e. space inhomogeneous, are few, as they combine the difficulty of getting explicit contraction rates for hypoelliptic diffusions as well as  a non linear term. Concerning the uniform in time propagation of chaos, there are no results except in the strictly convex case (with very small perturbation). We however refer to \cite{Villani_hypo} for a result on the torus with $W$ bounded with continuous derivative of all orders, see also \cite{BD95}. Based on functional inequalities (Poincar\'e or logarithmic Sobolev inequalities) for mean field models obtained in \cite{GLWZ19}, other results were obtained provided the confining potential is a small perturbation of a quadratic function as in \cite{M17, Guillin_Liu_kinetic, GM20} which combines the hypocoercivity approach with independent of the number of particles constants appearing in the logarithmic Sobolev inequalities. Our results generalize \cite{GM20}. Indeed, we may consider non gradient interactions whereas it is crucial in their approach to know explicitly the invariant measure of the particles system, and also we may handle only locally Lipschitz confinement potential, whereas they impose at most quadratic growth of the potentials, and non strictly convex at infinity potential. It is however difficult to compare the smallness of the interaction potentials needed in both approaches. Note however that they obtain convergence to equilibrium in entropy whereas we get it in Wasserstein distance (controlled by entropy through a Talagrand inequality). Using a coupling strategy, and more precisely synchronous couplings, results under strict convexity assumption were obtained in \cite{BGM10} for contraction rates in Wasserstein distance, see also \cite{kazeykina2020ergodicity} but only for the nonlinear system.

As we mentioned, we adapt a proof from \cite{Eberle_Guillin_Zimmer_Langevin}, which tackles \eqref{FP} without interaction term. The article uses a Lyapunov condition that guarantees the recurrence of the process on a compact set. This idea is common when proving similar results through a probabilistic lens (see for instance \cite{Talay2002StochasticHS} or \cite{BAKRY}). Lyapunov conditions may also help to implement hypocoercivity techniques {\it \`a la Villani} to handle entropic convergence for non quadratic potentials, see \cite{CGMZ19}. Under the assumption $U$ "greater than a quadratic function" at infinity and  $\nabla W$ Lipschitz continuous, we too consider a Lyapunov function that allows us to construct a specific semimetric improving the convergence speed. But, and this is to our knowledge something new, when proving propagation of chaos we add a form of non linearity in the quantity we consider to tackle a part of the non linearity appearing in the dynamic (see Section~\ref{Preuve_2} below).

%
%Section A modified distance
%

\section{Modified semimetrics}\label{section_2}

As mentioned in the introduction, the proofs rely on the construction of suitable semimetrics on $\R^{2d}$ and $\R^{2dN}$. They are introduced in this 
section, together with some useful properties. In all this section, $\lambda,A,\tilde A,L_U$ and $L_W$ are given by Assumptions \ref{HypU}, \ref{HypU_lip} and \ref{HypW} and Lemma~\ref{minU}. 

Before going into the details, let us highlight the main points of the construction of the semimetrics. It relies on the superposition of three ideas. The first idea is that, in order to deal with the kinetic process \eqref{FP}, the standard Euclidean norm $|x|^2+|v|^2$ is not suitable and one 
should consider a linear change of variables, like $(x,v) \mapsto (x ,x+\beta v)$ for some $\beta\in \R$. This is the case when using coupling methods as in \cite{Eberle_Guillin_Zimmer_Langevin,BGM10} but also when using hypocoercive modified entropies involving mixed derivatives as in \cite{Villani_hypo,Talay2002StochasticHS,B17,CGMZ19}, the link being made in \cite{M19}. This motivates the definition of $r$ below. The second idea is a modification of this distance $r$ by some concave function $f$, which is related to the fact we are using, at least in some parts of the space, a reflection coupling. The concavity is well adapted to It\^o's formula enabling the diffusion to provide a contraction effect (in a compact). This method has been considered for elliptic diffusions in \cite{Eberle_reflection}, see also \cite{Eberle_Guillin_Zimmer_Harris}. Intuitively, the contraction  is produced by the fact that a random decrease in $r$  has more effect on $f(r)$ than a random increase of the same amount. Finally, the third idea is the multiplication of a distance by a Lyapunov function $G$, which has first been used for Wasserstein distances in \cite{HM08}. That way, on average, $f(r)G$ tends to decay because, when $r$ is small, $f(r)$ tends to decay and, when $r$ is large, $G$ tends to decay.

\subsection{A Lyapunov function}

Let
\begin{equation}\label{eq:def_gamma_B}
\gamma=\frac{\lambda}{2\left(\lambda+1\right)}\,,\qquad B=24\left(A+\left(\lambda-\gamma\right)\tilde{A}+d\right)
\end{equation}
and, for $x,v\in \R^d$,
\begin{equation*}
H\left(x,v\right)=24U\left(x\right)+\left(6\left(1-\gamma \right)+\lambda\right)|x|^2+12x\cdot v+12|v|^2\,.
\end{equation*} 
For $\mu$ a probability measure on $\R^d$ with finite first moment, $\nabla W$ being assumed Lipschitz continuous, denote by $\mathcal{L}_\mu$ the generator given by 
\begin{equation*}
\mathcal{L}_\mu\phi\left(x,v\right)=v\cdot\nabla_x\phi\left(x,v\right)-\left(v+\nabla U\left(x\right)+\nabla W\ast\mu\left(x\right)\right)\cdot\nabla_v\phi\left(x,v\right)+\Delta_v\phi\left(x,v\right).
\end{equation*}
The main properties of $H$ are the following.
\begin{lemma}\label{Lya}
Under Assumptions \ref{HypU}, \ref{HypU_lip} and \ref{HypW}, for all $x,v\in\mathbb{R}^d$ and $\mu$,
\begin{align}\label{maj_par_H}
H\left(x,v\right)	\geq&24U(x)+\lambda|x|^2+12\left|v+\frac{x}{2}\right|^2,\\
\label{Lya_H}\mathcal{L}_\mu H\left(x,v\right)\leq &B+L_W(6+8\lambda)\left(\int |y|d\mu(y)\right)^2-\left(\frac{3}{4}\lambda+\lambda^2\right)|x|^2-\gamma  H\left(x,v\right),\\
\mathcal{L}_\mu H\left(x,v\right)\leq &B+\left(\left(\int |y|d\mu(y)\right)^2-|x|^2\right)\left(\frac{3}{4}\lambda+\lambda^2\right)-\gamma  H\left(x,v\right).
\end{align}
In particular $H$ is non-negative and goes to $+\infty$ at infinity.
\end{lemma}

The proof follows from elementary computations and is detailed in Appendix~\ref{preuve_lya}. Notice that  the condition $L_W\leq\lambda/8$ is used 
here.

%Remarque

In the case of particular interest where $\mu=\mu_t$ is given by \eqref{FP}, taking the expectation in \eqref{Lya_H}  and using Gronwall's lemma, we immediately get the following.

\begin{lemma}\label{majMoment}
Under Assumptions \ref{HypU}, \ref{HypU_lip} and \ref{HypW}, let $(X_t,V_t)_{t\geqslant0}$ be a  solution of \eqref{FP} with finite second moment at initial time. For all $t\geq0$,
\begin{align}
\label{E_H_Gronwall}
\frac{d}{dt}\mathbb{E}H\left(X_t,V_t\right)&\leq B-\gamma  \mathbb{E}H\left(X_t,V_t\right),\\
\label{gronwall_H}
\mathbb{E}H\left(X_t,V_t\right)-\frac{B}{\gamma }&\leq\left(\mathbb{E}H\left(X_0,V_0\right)-\frac{B}{\gamma }\right)e^{-\gamma  t}.
\end{align}
\end{lemma}

%
%Subsection Definition of the distance
%

\subsection{Change of variable and concave modification}\label{subsection_cstes}

We start by fixing the values of some parameters. The somewhat intricate expressions in this section are dictated by the computations arising in the proofs later on. Recall the definition of $\gamma$ and $B$ in \eqref{eq:def_gamma_B}.  Set 
 \[ \alpha=L_U+\frac{\lambda}{4}\,,\quad R_0=\sqrt{\frac{24B}{5\gamma \min\left(3,\frac{\lambda}{3}\right)}} 
\,,\quad R_1= \sqrt{\frac{24\left((1+\alpha)^2+\alpha^2\right) }{5\gamma \min\left(3,\frac{\lambda}{3}\right)}B}.\]
For $x,\tilde x,v,\tilde v\in \R^d$, set
\begin{align*}
r(x,\tilde{x},v,\tilde{v})&=\alpha|x-\tilde{x}|+|x-\tilde{x}+v-\tilde{v}|\,.
\end{align*}

%Lemme

\begin{lemma}\label{Prop-H}
Under Assumptions \ref{HypU}, \ref{HypU_lip} and \ref{HypW}, for all $x,\tilde{x},v,\tilde{v}\in \R^d$,  
\begin{equation}\label{r_min_rho}
r(x,\tilde{x},v,\tilde{v})^2\leq2\frac{\left(1+\alpha\right)^2+\alpha^2}{\min\left(\frac{1}{3}\lambda,3\right)}\left(H\left(x,v\right)+H\left(\tilde{x},\tilde{v}\right)\right)\,,
\end{equation}
so that, in particular,
\begin{equation*}
r(x,\tilde{x},v,\tilde{v})\geq R_1 \qquad \Rightarrow \qquad \gamma  H\left(x,v\right)+\gamma  H\left(\tilde{x},\tilde{v}\right)\geq\frac{12}{5}B.
\end{equation*}
\end{lemma}

We refer to Appendix~\ref{preuve_prop_de_H} for the proof. Let 
\begin{align}
c=&\min\Bigg\{\frac{\gamma }{36}\,,\,  \frac{B}{3}\,,\, \frac{1}{7}\min\left(\frac{1}{2}-\frac{L_U+L_W}{2\alpha},2\sqrt{\frac{L_U+L_W}{2\pi\alpha}}\right)\exp\left(-\frac{1}{8}\left(\frac{L_U+L_W}{\alpha}+\alpha+96\max\left(\frac{1}{2\alpha},1\right)\right)R_1^2\right)\Bigg\}\label{c_explicit}\,.
\end{align}
Set
\[\epsilon = \frac{3c}B\,,\qquad  \mathbf{C}=c+2\epsilon B\]
and, for $s\geq 0$,
\[\phi\left(s\right)=\exp\left(-\frac{1}{8}\left(\frac{1}{\alpha}\left(L_U+L_W\right)+\alpha+96\epsilon\max\left(\frac{1}{2\alpha},1\right)\right)s^2\right), \qquad \Phi\left(s\right)=\int_0^{s}\phi\left(u\right)du\]
\[g\left(s\right)=1-\frac{\mathbf{C}}{4}\int_0^{s}\frac{\Phi\left(u\right)}{\phi\left(u\right)} du \,,\qquad
f\left(s\right)=\int_0^{\min\left(s,R_1\right)}\phi\left(u\right)g\left(u\right)du.\]

\begin{remark}
The parameters above are far from being optimal. They are somewhat roughly chosen as we only wish to convey the fact that every constant is explicit.
\end{remark}

The next  lemma, proved in Appendix~\ref{choix_cstes}, gathers the intermediary bounds that will be useful in the proofs of the main results.

\begin{lemma}\label{Choix_constantes}
Under Assumptions \ref{HypU}, \ref{HypU_lip} and \ref{HypW},
\begin{align}
c\leq&\frac{\gamma }{6}\left(1-\frac{\frac{5\gamma }{6}}{2\epsilon B+ \frac{5\gamma }{6}}\right), \label{cond_region_3}\\
L_U+L_W<&\alpha, \label{cond_alpha}\\
c+2\epsilon B\leq&\frac{1}{2}\left(1-\frac{L_U+L_W}{\alpha}\right)\inf_{r\in]0,R_1]}\frac{r\phi\left(r\right)}{\Phi\left(r\right)},\label{cond_region_2_prop}\\
c+2\epsilon B\leq&2\left(\int_0^{R_1}\Phi\left(s\right)\phi\left(s\right)^{-1}ds\right)^{-1}, \label{cond_g_demi_prop}\\
\forall s\geq0,\qquad\ 0=&4 \phi'\left(s\right)+\left(\frac{1}{\alpha}\left(L_U+L_W\right)+\alpha+96\epsilon\max\left(\frac{1}{2\alpha},1\right)\right)s\phi\left(s\right).\label{def-phi}
\end{align}
\end{lemma}

The main properties of $f$ are the following.
%Lemme

\begin{lemma}\label{lem:f}
The function $f$ is 
$\mathcal{C}^2$ on $(0,R_1)$ with $f'_+\left(0\right)=1$ and $f'_-\left(R_1\right)>0$, and constant on $[R_1,\infty)$. Moreover, it is non-negative, non-decreasing and concave, and for all $s\geq 0$,
\begin{equation*}
\min\left(s,R_1\right)f'_-\left(R_1\right)\leq f\left(s\right)\leq\min\left(s,f\left(R_1\right)\right)\leq\min\left(s,R_1\right).
\end{equation*} 
\end{lemma}

%Preuve

\begin{proof}
First, notice that \eqref{cond_g_demi_prop} ensures that $g(s)\geq\frac{1}{2}$ for all $s\geq 0$. Then, all the points  immediately follow from the fact the functions $\phi$ and $g$ are $\mathcal{C}^2$, positive  and decreasing, with $\phi(0)=g(0)=1$.
\end{proof}

\subsection{The modified semimetrics}

For  $x,\tilde{x},v,\tilde{v}\in \R^d$,   set
\begin{align*}
G\left(x,v,\tilde{x},\tilde{v}\right)&=1+\epsilon H\left(x,v\right)+\epsilon H\left(\tilde{x},\tilde{v}\right),\\
\label{def_rho}
\rho\left(x,v,\tilde{x},\tilde{v}\right)&=f\left(r\left(x,v,\tilde{x},\tilde{v}\right)\right)G\left(x,v,\tilde{x},\tilde{v}\right).
\end{align*}

An immediate corollary of Lemmas~\ref{Prop-H} and \ref{lem:f} is that $\rho$ is a semimetric on $\R^{2d}$ which controls the usual L1 and L2 distances:
\begin{lemma}\label{rho_1_2}
There are explicit constants $\mathcal{C}_1,\mathcal{C}_2,\mathcal{C}_r, \mathcal{C}_z>0$ such that for all $x,x',v,v'\in\mathbb{R}^d$,
\begin{align*}
|x-x'|+|v-v'| & \leq \mathcal{C}_1\rho\left(\left(x,v\right),\left(x',v'\right)\right)\\
|x-x'|^2+|v-v'|^2 &\leq \mathcal{C}_2\rho\left(\left(x,v\right),\left(x',v'\right)\right)\\
r(x,v,x',v')&\leq \mathcal{C}_r\rho\left(\left(x,v\right),\left(x',v'\right)\right)\\
|x-x'|&\leq\mathcal{C}_z f(r(x,v,x',v'))\left(1+\epsilon \sqrt{H(x,v)}+\epsilon \sqrt{H(x',v')}\right).
\end{align*}
\end{lemma}
We also mention a technical lemma, see Appendix~\ref{preuve_dif_H} for proof.
\begin{lemma}\label{lemma_dif_H}
For all $x,v,\tilde{x},\tilde{v}\in\mathbb{R}^d$
\begin{equation}
|H\left(x,v\right)-H\left(\tilde{x},\tilde{v}\right)|\leq C_{dH,1}r(x,\tilde{x},v,\tilde{v})+C_{dH,2}r(x,\tilde{x},v,\tilde{v})\left(\sqrt{H\left(x,v\right)}+\sqrt{H\left(\tilde{x},\tilde{v}\right)}\right),
\end{equation}
where
\begin{align*}
C_{dH,1}:=\frac{24|\nabla U(0)|}{\alpha}\ \ \text{and}\ \ C_{dH,2}:=\frac{24L_U}{\alpha\sqrt{\lambda}}+\frac{6(1-\gamma)+\lambda-3}{\alpha\sqrt{\lambda}}+2\sqrt{3}\max\left(1,\frac{1}{2\alpha}\right).
\end{align*}
\end{lemma}

Finally, for $\mu$ and $\nu$ two probability measures on $\mathbb{R}^{2d}$ and a measurable function  $h:\mathbb{R}^{2d}\times\mathbb{R}^{2d}\rightarrow\mathbb{R}$, we define
\begin{equation*}
\mathcal{W}_{h}\left(\mu,\nu\right)=\inf_{\Gamma\in\Pi\left(\mu,\nu\right)}\int h\left(x,v,\tilde{x},\tilde{v}\right)\Gamma\left(d\left(x,v\right)d\left(\tilde{x},\tilde{v}\right)\right)\,.
\end{equation*}

%
%section
%

\section{Proof of Theorem~\ref{thm_1}}\label{Preuve_1}

In this section, for the sake of clarity, we only assume the potential $U$ satisfies Assumption~\ref{HypU} and Assumption~\ref{HypU_lip}. We refer 
to Appendix~\ref{Preuve_loc_lip} for the adjustment of the proof in the case $\nabla U$ locally Lipschitz continuous.

Our goal is to prove the following result

%Theoreme

\begin{theorem}\label{contraction_rho}
Let $\mathcal{C}^0>0$. Let $U\in\mathcal{C}^2\left(\mathbb{R}^{d}\right)$ satisfy Assumption~\ref{HypU} and Assumption~\ref{HypU_lip}. Let
\begin{align*}
\tilde{\mathcal{C}}_K:=\mathcal{C}_1\left(1+\frac{2\epsilon B}{\gamma }+2\epsilon \mathcal{C}^0\right)+2\epsilon\left(\frac{B}{\gamma }+\mathcal{C}^0\right)\frac{6+8\lambda}{\lambda}.
\end{align*}
For all $W\in\mathcal{C}^2\left(\mathbb{R}^{d}\right)$ satisfying Assumption~\ref{HypW} with $L_W<c/ \tilde{\mathcal{C}}_K$, for all probability measures $\nu^1_0$ and $\nu^2_0$ on $\mathbb{R}^{2d}$ satisfying $\mathbb{E}_{\nu^1_0}H\leq \mathcal{C}^0$ and $\mathbb{E}_{\nu^2_0}H\leq \mathcal{C}^0$
\begin{equation*}
\forall t\geq0,\hspace{0.5cm}\mathcal{W}_\rho\left(\nu^1_t,\nu^2_t\right)\leq e^{-(c-L_W\tilde{\mathcal{C}}_K) t}\mathcal{W}_\rho\left(\nu^1_0,\nu^2_0\right),
\end{equation*}
where $\nu^1_t$ (resp. $\nu^2_t$) is a solution of \eqref{Liouville} with 
initial distribution $\nu^1_0$ (resp. $\nu^2_0$).
\end{theorem}

%
%subsection
%

\subsection{Step one: Coupling and evolution of the coupling semimetric}
\label{sec:step_one}

Let $\xi>0$, and let $rc,\ sc:\mathbb{R}^{2d}\mapsto[0,1]$ be two Lipschitz continuous functions such that :
\begin{align*}
&rc^2+sc^2=1,\\
&rc\left(z,w\right)=0 \text{ if }|z+w|\leq\frac{\xi}{2}\text{ or }\alpha|z|+|z+w|\geq R_1+\xi,\\
&rc\left(z,w\right)=1 \text{ if }|z+w|\geq\xi\text{ and }\alpha|z|+|z+w|\leq R_1.
\end{align*}
These two functions translate into mathematical terms the regions in which we use a reflection coupling (represented by $rc=1$) and the ones where we use a synchronous coupling (represented by $sc=1$). Finally, $\xi$ is a parameter that will vanish to zero in the end. We therefore consider the following coupling :
\begin{equation*}
\left\{
    \begin{array}{ll}
        dX_t=V_tdt\\
        dV_t=-V_tdt-\nabla U\left(X_t\right)dt-\nabla W\ast\mu_t\left(X_t\right)dt+\sqrt{2}rc\left(Z_t,W_t\right)dB^{rc}_t+\sqrt{2}sc\left(Z_t,W_t\right)dB^{sc}_t\\
        \mu_t=\text{Law}\left(X_t\right)\\
        d\tilde{X}_t=\tilde{V}_tdt\\
        d\tilde{V}_t=-\tilde{V}_tdt-\nabla U(\tilde{X}_t)dt-\nabla W\ast\tilde{\mu}_t(\tilde{X}_t)dt+\sqrt{2}rc\left(Z_t,W_t\right)\left(Id-2e_te^T_t\right)dB^{rc}_t\\\qquad\qquad+\sqrt{2}sc\left(Z_t,W_t\right)dB^{sc}_t\\
        \tilde{\mu}_t=\text{Law}(\tilde{X}_t),
    \end{array}
\right.
\end{equation*}
where $B^{rc}$ and $B^{sc}$ are independent Brownian motions, and
\begin{align*}
Z_t=X_t-\tilde{X}_t,\ \ \ W_t=V_t-\tilde{V}_t,\ \ \ Q_t=Z_t+W_t,\ \ 
\ e_t=\left\{
    \begin{array}{ll}
    \frac{Q_t}{|Q_t|}&\text{ if }Q_t\neq0,\\
    0&\text{ otherwise,}
    \end{array}
\right.
\end{align*}
and $e^T_t$ is the transpose of $e_t$. Then
\begin{equation}\label{dynZ_t}
\frac{dZ_t}{dt}=W_t=Q_t-Z_t.
\end{equation}
So $\frac{d|Z_t|}{dt}=\frac{Z_t}{|Z_t|}\left(Q_t-Z_t\right)$ for every $t$ such that $Z_t\neq0$,
and $ \frac{d|Z_t|}{dt}\leq|Q_t|$  for every $t$ such that $Z_t=0$.
In particular
\begin{equation*}
\frac{d|Z_t|}{dt}\leq|Q_t|-|Z_t|.
\end{equation*}
We start by using It\^o's formula to compute the evolution of $|Q_t|$. The following lemma is Lemma 7 of A.~Durmus \textit{et al.} \cite{Uniform_Prop_Chaos} of which, for the sake of completeness, we give the proof.

%Lemme

\begin{lemma}\label{dyn_Q_t}
Under Assumption~\ref{HypU}, Assumption~\ref{HypU_lip} and Assumption~\ref{HypW}, we have almost surely for all $t\geq0$.
 \begin{align}
 d|Q_t|=&-e_t\cdot \left(\nabla U\left(X_t\right)-\nabla U(\tilde{X}_t)\right)dt-e_t\cdot\left(\nabla W\ast\mu_t\left(X_t\right)-\nabla W\ast\tilde{\mu}_t(\tilde{X}_t)\right)dt\label{|Q_t|}\\
 &+2\sqrt{2}rc\left(Z_t,W_t\right)e_t\cdot dB^{rc}_t\nonumber
  \end{align}
\end{lemma}
%Preuve

\begin{proof}
Let $t\geq0$.
We begin by considering the dynamics of $Z_t$, $W_t$ and $Q_t$. We have
 \begin{align*}
 &dZ_t=W_tdt\\
 &dW_t=-W_tdt-\left(\nabla U\left(X_t\right)-\nabla U(\tilde{X}_t)\right)dt-\left(\nabla W\ast\mu_t\left(X_t\right)-\nabla W\ast\tilde{\mu}_t(\tilde{X}_t)\right)dt\\
 &\hspace{1cm}+2\sqrt{2}rc\left(Z_t,W_t\right)e_te_t\cdot dB^{rc}_t\\
 &dQ_t=-\left(\nabla U\left(X_t\right)-\nabla U(\tilde{X}_t)\right)dt-\left(\nabla W\ast\mu_t\left(X_t\right)-\nabla W\ast\tilde{\mu}_t(\tilde{X}_t)\right)dt+2\sqrt{2}rc\left(Z_t,W_t\right)e_te_t\cdot dB^{rc}_t.
 \end{align*}
Therefore
 \begin{align*}
d|Q_t|^2&=-2Q_t\cdot\left(\nabla U\left(X_t\right)-\nabla U(\tilde{X}_t)\right)dt-2Q_t\cdot\left(\nabla W\ast\mu_t\left(X_t\right)-\nabla W\ast\tilde{\mu}_t(\tilde{X}_t)\right)dt\\
&+4\sqrt{2}rc\left(Z_t,W_t\right)\left(Q_t\cdot e_t\right)e_t\cdot dB^{rc}_t+8 rc^2\left(Z_t,W_t\right)dt.
 \end{align*}
We consider, for $\eta>0$, the function $\psi_\eta\left(r\right)=\left(r+\eta\right)^{1/2}$ which is $\mathcal{C}^\infty$ on $]0,\infty[$ and satisfies
 \begin{align*}
 \forall r\geq0,\hspace{0.3cm}&\lim_{\eta\rightarrow0}\psi_\eta\left(r\right)=r^{1/2},\ \ \ \lim_{\eta\rightarrow0}2\psi'_\eta\left(r\right)=r^{-1/2},\ \ \ \lim_{\eta\rightarrow0}4\psi''_\eta\left(r\right)=-r^{-3/2},\\
 &\text{ and thus }\lim_{\eta\rightarrow0}2r\psi''_\eta\left(r\right)+\psi'_\eta\left(r\right)=0.
  \end{align*}
Then
 \begin{align*}
d\psi_\eta\left(|Q_t|^2\right)=&-2\psi'_\eta\left(|Q_t|^2\right)Q_t\cdot\left(\nabla U\left(X_t\right)-\nabla U(\tilde{X}_t)\right)dt\\
&-2\psi'_\eta\left(|Q_t|^2\right)Q_t\cdot\left(\nabla W\ast\mu_t\left(X_t\right)-\nabla W\ast\tilde{\mu}_t(\tilde{X}_t)\right)dt\\
&+4\psi'_\eta\left(|Q_t|^2\right)\sqrt{2}rc\left(Z_t,W_t\right)\left(Q_t\cdot e_t\right)e_t\cdot dB^{rc}_t+8\psi'_\eta\left(|Q_t|^2\right) rc^2\left(Z_t,W_t\right)dt\\
&+16\psi''_\eta\left(|Q_t|^2\right)rc^2\left(Z_t,W_t\right)|Q_t|^2dt.
 \end{align*}
We make sure each individual term converges almost surely as $\eta\rightarrow0$. First, we notice that $$2|Q_t|\psi'_\eta\left(|Q_t|^2\right)=\frac{|Q_t|}{\left(|Q_t|^2+\eta\right)^{1/2}}\leq1.$$ So
  \begin{align*}
 2\psi'_\eta\left(|Q_t|^2\right)Q_t\cdot\left(\nabla U\left(X_t\right)-\nabla U(\tilde{X}_t)\right)\leq|\nabla U\left(X_t\right)-\nabla U(\tilde{X}_t)|\leq L_U|Z_t|.
 \end{align*}
 Then, by dominated convergence, for all $T\geq0$  almost surely
  \begin{align*}
  \lim_{\eta\rightarrow0}\int_{0}^T2\psi'_\eta\left(|Q_t|^2\right)Q_t\cdot\left(\nabla U\left(X_t\right)-\nabla U(\tilde{X}_t)\right)dt&=\int_{0}^T\frac{Q_t}{|Q_t|}\cdot\left(\nabla U\left(X_t\right)-\nabla U(\tilde{X}_t)\right)dt\\
  &=\int_{0}^Te_t \cdot\left(\nabla U\left(X_t\right)-\nabla U(\tilde{X}_t)\right)dt.
  \end{align*}
  Likewise for all  $T\geq0$
 \begin{align*}
 2\psi'_\eta\left(|Q_t|^2\right)Q_t\cdot\left(\nabla W\ast\mu_t\left(X_t\right)-\nabla W\ast\tilde{\mu}_t(\tilde{X}_t)\right)&\leq\left|\nabla W\ast\mu_t\left(X_t\right)-\nabla W\ast\tilde{\mu}_t(\tilde{X}_t)\right|\\&
 \leq L_W|Z_t|+L_W\mathbb{E}|Z_t|,
 \end{align*}
 hence
\begin{align*}
&\lim_{\eta\rightarrow0}\int_{0}^T2\psi'_\eta\left(|Q_t|^2\right)Q_t\cdot\left(\nabla W\ast\mu_t\left(X_t\right)-\nabla W\ast\tilde{\mu}_t(\tilde{X}_t)\right)dt\\&\qquad\qquad=\int_{0}^Te_t\cdot\left(\nabla W\ast\mu_t\left(X_t\right)-\nabla W\ast\tilde{\mu}_t(\tilde{X}_t)\right)dt.
\end{align*}   
Then, since $rc\left(Z_t,W_t\right)=0$ for $|Q_t|\leq\frac{\xi}{2}$ and 

  \begin{align*}
8\psi'_\eta\left(|Q_t|^2\right)+16\psi''_\eta\left(|Q_t|^2\right)|Q_t|^2&=4\left(\frac{1}{\left(|Q_t|^2+\eta\right)^{1/2}}-\frac{|Q_t|^2}{\left(|Q_t|^2+\eta\right)^{3/2}}\right)=4\frac{\eta}{\left(|Q_t|^2+\eta\right)^{3/2}}\leq\frac{4\eta}{|Q_t|^3},
  \end{align*}
we have by dominated convergence
\begin{align*}
\lim_{\eta\rightarrow0}\int_{0}^T\left(8d\psi'_\eta\left(|Q_t|^2\right) rc^2\left(Z_t,W_t\right)+16\psi''_\eta\left(|Q_t|^2\right)rc^2\left(Z_t,W_t\right)|Q_t|^2\right)dt=0.
\end{align*}
Finally, by Theorem 2.12 chapter 4 of \cite{Revuz_Yor}
\begin{align*}
\lim_{\eta\rightarrow0}\int_{0}^T4\sqrt{2}\psi'_\eta\left(|Q_t|^2\right)rc\left(Z_t,W_t\right)\left(Q_t\cdot e_t\right)e_t\cdot dB^{rc}_t=\int_{0}^T2\sqrt{2} rc\left(Z_t,W_t\right)e_t\cdot dB^{rc}.
\end{align*}
For any $t$, we obtain the desired result almost surely. The continuity of $t\mapsto|Q_t|$ then allows us to conclude that \eqref{|Q_t|} is almost surely true for all $t$.
\end{proof}

We denote
\begin{align}
r_t&:=\alpha|X_t-\tilde{X}_t|+|X_t-\tilde{X}_t+V_t-\tilde{V}_t|=\alpha|Z_t|+|Q_t|,\\
\rho_t&:=f\left(r_t\right)G_t \text{ where }G_t=1+\epsilon H\left(X_t,V_t\right)+\epsilon  H(\tilde{X}_t,\tilde{V}_t).
\end{align}
Since $H\left(x,v\right)\geq0$ we have $G_t\geq1$. We now state the main lemma of this section.

%Lemme

\begin{lemma}\label{lemme-maj-rho}
Under Assumption~\ref{HypU}, Assumption~\ref{HypU_lip} and Assumption~\ref{HypW}, let $c\in]0,\infty[$. Then almost surely for all $t\geq0$
\begin{equation}\label{maj-rho}
\forall t\geq0,\ e^{ct}\rho_t\leq \rho_0+\int_0^te^{cs}K_sds+M_t,
\end{equation}
where $\left(M_t\right)_t$ is a continuous local martingale and
\begin{align*}
K_t&=4  f''\left(r_t\right)rc\left(Z_t,W_t\right)^2G_t+c f\left(r_t\right)G_t+\left(\alpha\frac{d|Z_t|}{dt}+ \left(L_U+L_W\right)|Z_t|+L_W\mathbb{E}|Z_t|\right)f'\left(r_t\right)G_t\\
&+\epsilon\left(2B-\gamma  H\left(X_t,V_t\right)-\gamma H(\tilde{X}_t,\tilde{V}_t)\right)f\left(r_t\right)+\epsilon L_W\left(6+8\lambda\right)\left(\mathbb{E}\left(|X_t|\right)^2+\mathbb{E}(|\tilde{X}_t|)^2\right)f\left(r_t\right)\\
&+96\epsilon\max\left(1,\frac{1}{2\alpha}\right) r_tf'\left(r_t\right)rc\left(Z_t,W_t\right)^2.
\end{align*}
\end{lemma}

%Preuve

\begin{proof}

Using \eqref{|Q_t|}
\begin{align*}
|Q_t|&=|Q_0|+A_t^Q+M_t^Q\ \text{ with }\\
&dA_t^Q=-e_t\cdot\left(\nabla U\left(X_t\right)-\nabla U(\tilde{X}_t)\right)dt-e_t\cdot\left(\nabla W\ast\mu_t\left(X_t\right)-\nabla W\ast\tilde{\mu}_t(\tilde{X}_t)\right)dt\\
&dM_t^Q=2\sqrt{2}rc\left(Z_t,W_t\right)e_t\cdot dB^{rc}_t.
\end{align*}
Therefore
$r_t=|Q_0|+\alpha|Z_t|+A_t^Q+ M_t^Q.
$ Let $c>0$. By It\^o's formula
\begin{align*}
d\left(e^{ct}f\left({r_t}\right)\right)=ce^{ct}f\left(r_t\right)dt+e^{ct}f'\left(r_t\right)dr_t+\frac{1}{2}e^{ct}f''\left(r_t\right)8rc^2\left(Z_t,W_t\right)dt.
\end{align*}
Hence
\begin{align*}
e^{ct}f\left(r_t\right)=&f\left(r_0\right)+\hat{A}_t+\hat{M}_t\text{ with }\\
d\hat{A}_t=&\Big(c f\left(r_t\right)+\alpha f'\left(r_t\right)\frac{d|Z_t|}{dt}- f'\left(r_t\right)e_t\cdot\left(\nabla U\left(X_t\right)-\nabla 
U(\tilde{X}_t)\right)\\
&-f'\left(r_t\right)e_t\cdot\left(\nabla W\ast\mu_t\left(X_t\right)-\nabla W\ast\tilde{\mu}_t(\tilde{X}_t)\right)+4 f''\left(r_t\right)rc^2\left(Z_t,W_t\right)\Big)e^{ct}dt\\
d\hat{M}_t=&e^{ct}2\sqrt{2} f'\left(r_t\right)rc\left(Z_t,W_t\right)e_t\cdot dB^{rc}_t.
\end{align*}
We now consider the evolution of $G_t=1+\epsilon H\left(X_t,V_t\right)+\epsilon  H(\tilde{X}_t,\tilde{V}_t)$
\begin{align*}
dG_t=&\epsilon\left(\mathcal{L}_{\mu_t}H\left(X_t,V_t\right)+\mathcal{L}_{\tilde{\mu}_t} H(\tilde{X}_t,\tilde{V}_t)\right)dt\\
&+\epsilon\sqrt{2}rc\left(Z_t,W_t\right)\left(\nabla_vH\left(X_t,V_t\right)-\nabla_v H(\tilde{X}_t,\tilde{V}_t)\right)\cdot e_te_t^T dB^{rc}_t\\
&+\epsilon\sqrt{2}rc\left(Z_t,W_t\right)\left(\nabla_vH\left(X_t,V_t\right)+\nabla_v H(\tilde{X}_t,\tilde{V}_t)\right)\cdot\left(Id-e_te_t^T\right)dB^{rc}_t\\
&+\epsilon\sqrt{2}sc\left(Z_t,W_t\right)\left(\nabla_vH\left(X_t,V_t\right)+\nabla_v H(\tilde{X}_t,\tilde{V}_t)\right)\cdot dB^{sc}_t.
\end{align*}
Therefore
$e^{ct}\rho_t=e^{ct}f\left(r_t\right)G_t=\rho_0+A_t+M_t,
$ where
\begin{align*}
dA_t=&G_td\hat{A}_t+\epsilon e^{ct}f\left(r_t\right)\left(\mathcal{L}_{\mu_t}H\left(X_t,V_t\right)+\mathcal{L}_{\tilde{\mu}_t}H(\tilde{X}_t,\tilde{V}_t)\right)dt\\
&+4\epsilon e^{ct}f'\left(r_t\right)rc^2\left(Z_t,W_t\right)\left(\nabla_vH\left(X_t,V_t\right)-\nabla_v H(\tilde{X}_t,\tilde{V}_t)\right)\cdot e_tdt,
\end{align*}
and $M_t$ is a continuous local martingale. This last equality uses the fact that $B^{rc}$ and $B^{sc}$ are independent Brownian motion and that $e_t\cdot\left(Id-e_te_t^T\right)=0$. Furthermore
\begin{align*}
|\nabla_vH\left(X_t,V_t\right)-\nabla_v H(\tilde{X}_t,\tilde{V}_t)|&=12|X_t+2V_t-\tilde{X}_t-2\tilde{V}_t|=12|2Q_t-Z_t|\\
&\leq24\left(\frac{1}{2}|Z_t|+|Q_t|\right)\\
&\leq24\max\left(1,\frac{1}{2\alpha}\right)r_t,
\end{align*}
so that
$dA_t\leq e^{ct}\tilde{K}_tdt,
$ where
\begin{align*}
\tilde{K}_t=&\Big(c f\left(r_t\right)+\alpha f'\left(r_t\right)\frac{d|Z_t|}{dt}-f'\left(r_t\right)e_t\cdot\left(\nabla U\left(X_t\right)-\nabla 
U(\tilde{X}_t)\right)\\
&- f'\left(r_t\right)e_t\cdot\left(\nabla W\ast\mu_t\left(X_t\right)-\nabla W\ast\tilde{\mu}_t(\tilde{X}_t)\right)+4f''\left(r_t\right)rc^2\left(Z_t,W_t\right)\Big)G_t\\
&+\epsilon\left(\mathcal{L}_tH\left(X_t,V_t\right)+\mathcal{L}_t H(\tilde{X}_t,\tilde{V}_t)\right)f\left(r_t\right)+96\epsilon\max\left(1,\frac{1}{2\alpha}\right) r_tf'\left(r_t\right)rc^2\left(Z_t,W_t\right).
\end{align*}
And we conclude using Lemma~\ref{majNablaW}, and Lemma~\ref{Lya}.
\end{proof}

%
%Subection Contraction
%

\subsection{Step two : Contractivity in various regions of space}\label{contraction}

At this point, we have
\begin{equation*}
\forall t\geq0,\ e^{ct}\rho_t\leq \rho_0+\int_0^te^{cs}K_sds+M_t,
\end{equation*}
where $M_t$ is a continuous local martingale and, by regrouping the terms 
according to how we will use them
\begin{align}
K_t&=\left(c f\left(r_t\right)+\left(\alpha \frac{d|Z_t|}{dt}+ (L_U+ L_W)|Z_t|\right)f'\left(r_t\right)\right)G_t\label{K_t_lip_1}\\
&+4\left( f''\left(r_t\right)G_t+24\epsilon\max\left(1,\frac{1}{2\alpha}\right) r_tf'\left(r_t\right)\right)rc\left(Z_t,W_t\right)^2\label{K_t_lip_1_bis}\\
&+\epsilon\left(2B-\gamma  H\left(X_t,V_t\right)-\gamma  H(\tilde{X}_t,\tilde{V}_t)\right)f\left(r_t\right)\label{K_t_lip_2}\\
&+L_Wf'\left(r_t\right)\mathbb{E}\left(|Z_t|\right)G_t+\epsilon L_W\left(6+8\lambda\right)\left(\mathbb{E}\left(|X_t|\right)^2+\mathbb{E}(|\tilde{X}_t|)^2\right)f\left(r_t\right).\label{K_t_lip_3}
\end{align}
Briefly,
\begin{itemize}
\item  lines \eqref{K_t_lip_1} and \eqref{K_t_lip_1_bis} will be non positive thanks to the construction of the function $f$ when using the reflection coupling,
\item  when only using the synchronous coupling, i.e when the deterministic drift is contracting, line \eqref{K_t_lip_1} alone will be sufficiently small,
\item  line \eqref{K_t_lip_2} translates the effect the Lyapunov function 
has in bringing back processes that would have ventured at infinity,
\item  finally, line \eqref{K_t_lip_3} contains the non linearity and will be tackled by taking $L_W$ sufficiently small.
\end{itemize}

In this section, we thus prove the following lemma
\begin{lemma}\label{maj_K}
Assume the confining potential $U$ satisfies Assumption~\ref{HypU} and Assumption~\ref{HypU_lip}, there is a constant $c^W>0$ such that for all interaction potential $W$ satisfying Assumption~\ref{HypW} with $L_W<c^W$, the following holds for $K_t$ defined in \eqref{K_t_lip_1}-\eqref{K_t_lip_3}
\begin{equation*}
\mathbb{E}K_t\leq\left(1+\alpha\right)\xi \mathbb{E}G_t+L_W\left(\mathcal{C}_K+\mathcal{C}^0_Ke^{-\gamma t}\right)\mathbb{E}\rho_t,
\end{equation*}
with
\begin{equation}\label{C_k}
\mathcal{C}_K=\mathcal{C}_1\left(1+\frac{2\epsilon B}{\gamma }\right)+\frac{2\epsilon B}{\gamma\lambda }\left(6+8\lambda\right),
\end{equation}
\begin{equation*}
\mathcal{C}^0_K=\epsilon\left(\mathcal{C}_1+\frac{6+8\lambda}{\lambda}\right)\left(\mathbb{E}H\left(X_0,V_0\right)+\mathbb{E}H(\tilde{X}_0,\tilde{V}_0)\right).
\end{equation*}
The constant $c^W$ is explicit, as it will be shown in Appendix~\ref{choix_cstes}.
\end{lemma}

To this end, we divide the space into three regions 
\begin{align*}
\mbox{Reg}_1=&\left\{(X_t,V_t,\tilde{X}_t,\tilde{V}_t)\text{ s.t. }|Q_t|\geq\xi \text{ and } r_t\leq R_1\right\},\\
\mbox{Reg}_2=&\left\{(X_t,V_t,\tilde{X}_t,\tilde{V}_t)\text{ s.t. }|Q_t|<\xi \text{ and } r_t\leq R_1\right\},\\
\mbox{Reg}_3=&\left\{(X_t,V_t,\tilde{X}_t,\tilde{V}_t)\text{ s.t. }r_t> 
R_1\right\},
\end{align*}
and consider
\begin{align*}
\mathbb{E}K_t=\mathbb{E}(K_t\mathds{1}_{\mbox{Reg}_1})+\mathbb{E}(K_t\mathds{1}_{\mbox{Reg}_2})+\mathbb{E}(K_t\mathds{1}_{\mbox{Reg}_3}).
\end{align*}

%Region 2-1

\subsubsection{First region : $|Q_t|\geq\xi$ and $r_t\leq R_1$}\label{first_region}

In this region of space, we use the Brownian motion through the reflection coupling and the construction of the function $f$ to bring the processes closer together.  Here we have $rc\left(Z_t,W_t\right)=1$. Recall $\alpha|Z_t|+|Q_t|=r_t$ and $G_t\geq1$.

\begin{itemize}
\item  We have
\begin{align*}
\alpha\frac{d|Z_t|}{dt}+ \left(L_U+L_W\right)|Z_t|&\leq\alpha|Q_t|-\alpha|Z_t|+\left(L_U+L_W\right)|Z_t|\\
&=\alpha r_t-\alpha^2|Z_t|-\alpha|Z_t|+  \left(L_U+L_W\right)|Z_t|\\
&\leq\left( \frac{1}{\alpha}\left(L_U+L_W\right)+\alpha \right)r_t.
\end{align*}

\item Since $G_t=1+\epsilon H\left(X_t,V_t\right)+\epsilon H(\tilde{X}_t,\tilde{V}_t)\geq1$,
\begin{equation}\label{def_C}
c G_t+\epsilon\left(2B-\gamma  H\left(X_t,V_t\right)-\gamma H(\tilde{X}_t,\tilde{V}_t)\right)\leq cG_t+2\epsilon BG_t=\mathbf{C}G_t.
\end{equation}
\item We then have, by \eqref{def-phi},
\begin{align*}
4 \phi'\left(r_t\right)+\left(\frac{1}{\alpha}\left(L_U+L_W\right)+\alpha+96\epsilon\max\left(\frac{1}{2\alpha},1\right)\right)r_t\phi\left(r_t\right)=0.
\end{align*}
Hence
\begin{align*}
4f''\left(r_t\right)&+\left(\frac{1}{\alpha}\left(L_U+L_W\right)+\alpha+96\epsilon\max\left(\frac{1}{2\alpha},1\right)\right)r_tf'\left(r_t\right)\\
&=4 \phi'\left(r_t\right)g\left(r_t\right)+4 \phi\left(r_t\right)g'\left(r_t\right)+\left( \frac{1}{\alpha}\left(L_U+L_W\right)+\alpha+96\epsilon\max\left(\frac{1}{2\alpha},1\right)\right)r_t\phi\left(r_t\right)g\left(r_t\right)\\
&=4\phi\left(r_t\right)g'\left(r_t\right),
\end{align*}
and
\begin{align*}
4 \phi\left(r_t\right)g'\left(r_t\right)+\mathbf{C}f\left(r_t\right)\leq-4 \frac{\mathbf{C}}{4}\Phi\left(r_t\right)+\mathbf{C}\Phi\left(r_t\right)=0.
\end{align*}
\item At this point, through this choice of function $f$, we are left with
\begin{align*}
K_t\mathds{1}_{\mbox{Reg}_1}\leq L_Wf'\left(r_t\right)\mathbb{E}\left(|Z_t|\right)G_t+\epsilon L_W\left(6+8\lambda\right)\left(\mathbb{E}\left(|X_t|\right)^2+\mathbb{E}(|\tilde{X}_t|)^2\right)f\left(r_t\right).
\end{align*}
Using Lemma~\ref{rho_1_2}, $f'\left(r_t\right)\leq1$ and $G_t\geq1$,
\begin{align*}
\mathbb{E}\left(K_t\mathds{1}_{\mbox{Reg}_1}\right)\leq L_W\mathcal{C}_1\mathbb{E}\left(\rho_t\right)\mathbb{E}\left(G_t\right)+\epsilon L_W\left(6+8\lambda\right)\left(\mathbb{E}\left(|X_t|\right)^2+\mathbb{E}(|\tilde{X}_t|)^2\right)\mathbb{E}\left(\rho_t\right).
\end{align*}
Recall \eqref{gronwall_H}
\begin{align*}
\mathbb{E}\left(G_t\right)=&1+\epsilon \mathbb{E}H\left(X_t,V_t\right)+\epsilon \mathbb{E}H(\tilde{X}_t,\tilde{V}_t),\\
\leq&1+\frac{2\epsilon B}{\gamma }+\epsilon\left(\mathbb{E}H\left(X_0,V_0\right)+\mathbb{E}H(\tilde{X}_0,\tilde{V}_0)\right)e^{-\gamma  t},
\end{align*}
and, since $H(x,v)\geq\lambda|x|^2$,
\begin{align*}
\mathbb{E}\left(|X_t|\right)^2+\mathbb{E}(|\tilde{X}_t|)^2\leq&\frac{1}{\lambda}\mathbb{E}H\left(X_t,V_t\right)+\frac{1}{\lambda}\mathbb{E}H(\tilde{X}_t,\tilde{V}_t),\\
\leq&\frac{2 B}{\gamma\lambda }+\frac{1}{\lambda}\left(\mathbb{E}H\left(X_0,V_0\right)+\mathbb{E}H(\tilde{X}_0,\tilde{V}_0)\right)e^{-\gamma  t}.
\end{align*}
Hence
\begin{align*}
\mathbb{E}\left(K_t\mathds{1}_{\mbox{Reg}_1}\right)\leq& L_W\left(\mathcal{C}_1\left(1+\frac{2\epsilon B}{\gamma }\right)+\frac{2\epsilon B}{\gamma\lambda }\left(6+8\lambda\right)\right)\mathbb{E}\left(\rho_t\right)\\
&+L_W\epsilon\left(\mathcal{C}_1+\frac{6+8\lambda}{\lambda}\right)\left(\mathbb{E}H\left(X_0,V_0\right)+\mathbb{E}H(\tilde{X}_0,\tilde{V}_0)\right)\mathbb{E}\left(\rho_t\right)e^{-\gamma  t}.
\end{align*}
We thus obtain
$\mathbb{E}\left(K_t\mathds{1}_{\mbox{Reg}_1}\right)\leq L_W\left(\mathcal{C}_K+\mathcal{C}^0_Ke^{-\gamma t}\right)\mathbb{E}\rho_t.
$
\end{itemize}

%Region 2

\subsubsection{Second region : $|Q_t|<\xi$ and $r_t\leq R_1$}

In this region of space, we use the naturally contracting deterministic drift thanks to a synchronous coupling. Here $R_1\geq r_t\geq \alpha|Z_t|\geq r_t-\xi$ so that
\begin{align*}
K_t&\leq\mathbf{C} f\left(r_t\right)G_t+\left(\alpha\xi-r_t+\xi+ \frac{1}{\alpha}\left(L_U+L_W\right)r_t\right)f'\left(r_t\right)G_t \\
&\quad+\left(4 f''\left(r_t\right)G_t+96\epsilon\max\left(\frac{1}{2\alpha},1\right) r_tf'\left(r_t\right)\right)rc\left(Z_t,W_t\right)^2\\
&\quad+L_Wf'\left(r_t\right)\mathbb{E}\left(|Z_t|\right)G_t+\epsilon L_W\left(6+8\lambda\right)\left(\mathbb{E}\left(|X_t|\right)^2+\mathbb{E}(|\tilde{X}_t|)^2\right)f\left(r_t\right),
\end{align*}
where we use \eqref{def_C}. First
\begin{align*}
4 f''\left(r_t\right)G_t+96\epsilon\max\left(\frac{1}{2\alpha},1\right) r_tf'\left(r_t\right)\leq0.
\end{align*}
We use \eqref{cond_alpha} to obtain, since $f\left(r_t\right)\leq\Phi\left(r_t\right)$ and $\frac{1}{2}\phi\left(r_t\right)\leq f'\left(r_t\right)=\phi\left(r_t\right)g\left(r_t\right)\leq \phi\left(r_t\right)$ by \eqref{cond_g_demi_prop},
\begin{align*}
K_t\leq&\xi\left(1+\alpha\right)\phi\left(r_t\right)g\left(r_t\right)G_t +G_t\left(\mathbf{C}\Phi\left(r_t\right)+\frac{1}{2}\left(\frac{1}{\alpha 
}\left(L_U+L_W\right)-1\right)r_t\phi\left(r_t\right)\right)\\
&+L_Wf'\left(r_t\right)\mathbb{E}\left(|Z_t|\right)G_t+\epsilon L_W\left(6+8\lambda\right)\left(\mathbb{E}\left(|X_t|\right)^2+\mathbb{E}(|\tilde{X}_t|)^2\right)f\left(r_t\right).
\end{align*}
Then, thanks to \eqref{cond_region_2_prop}, like in the first region of space
\begin{align*}
K_t\mathds{1}_{\mbox{Reg}_2}\leq \xi\left(1+\alpha\right)\phi\left(r_t\right)g\left(r_t\right)G_t+ L_Wf'\left(r_t\right)\mathbb{E}\left(|Z_t|\right)G_t+\epsilon L_W\left(6+8\lambda\right)\left(\mathbb{E}\left(|X_t|\right)^2+\mathbb{E}(|\tilde{X}_t|)^2\right)f\left(r_t\right).
\end{align*}
Hence, since $\phi\left(r_t\right)g\left(r_t\right)\leq1$

\begin{equation*}
\mathbb{E}K_t\mathds{1}_{\mbox{Reg}_2}\leq \xi\left(1+\alpha\right)\mathbb{E}\left(G_t\right)+L_W\left(\mathcal{C}_K+\mathcal{C}^0_Ke^{-\gamma t}\right)\mathbb{E}\rho_t.
\end{equation*}

%Region 2-3

\subsubsection{Third region : $r_t> R_1$}\label{region_3}

In this region, we use the Lyapunov function. Here $f'\left(r_t\right)=f''\left(r_t\right)=0$ so that 
\begin{align*}
K_t\mathds{1}_{\mbox{Reg}_3}=&\left(cG_t +\epsilon\left(2B-\gamma  H\left(X_t,V_t\right)-\gamma  H(\tilde{X}_t,\tilde{V}_t)\right)\right)f\left(r_t\right)\mathds{1}_{\mbox{Reg}_3}\\
&+\epsilon L_W\left(6+8\lambda\right)\left(\mathbb{E}\left(|X_t|\right)^2+\mathbb{E}(|\tilde{X}_t|)^2\right)f\left(r_t\right)\mathds{1}_{\mbox{Reg}_3}\\
=&\Big[\epsilon \left(c-\gamma \right)\left(H\left(X_t,V_t\right)+H(\tilde{X}_t,\tilde{V}_t)\right)+2\epsilon B+ c \\
&+\epsilon L_W\left(6+8\lambda\right)\left(\mathbb{E}\left(|X_t|\right)^2+\mathbb{E}(|\tilde{X}_t|)^2\right)\Big]f\left(r_t\right)\mathds{1}_{\mbox{Reg}_3}
\end{align*}
Since $c-\gamma <0$ as a consequence of \eqref{cond_region_3}, and using Lemma~\ref{Prop-H}
\begin{align*}
K_t&\leq\left( \left(c-\gamma \right)\epsilon\frac{12}{5}\frac{B}{\gamma }+2\epsilon B +c \right)f\left(r_t\right)+\epsilon L_W\left(6+8\lambda\right)\left(\mathbb{E}\left(|X_t|\right)^2+\mathbb{E}(|\tilde{X}_t|)^2\right)f\left(r_t\right)\\
&\leq\left(c \left(\frac{12\epsilon}{5}\frac{B}{\gamma }+1\right) -\frac{2}{5}\epsilon B\right)f\left(r_t\right)+\epsilon L_W\left(6+8\lambda\right)\left(\mathbb{E}\left(|X_t|\right)^2+\mathbb{E}(|\tilde{X}_t|)^2\right)f\left(r_t\right).
\end{align*}
Then, using \eqref{cond_region_3}, $\mathbb{E}K_t\mathds{1}_{\mbox{Reg}_3}\leq L_W\mathcal{C}_K\mathbb{E}\rho_t+L_W\mathcal{C}^0_K\mathbb{E}\rho_te^{-\gamma  t}$.

%
%Section Convergence
%

\subsection{Step three : Convergence}

Let $\Gamma$ be a coupling of $\nu^1_0$ and $\nu^2_0$ such that $\mathbb{E}_{\Gamma}\rho\left((x,v),(\tilde{x},\tilde{v}))\right)<\infty$. We consider the coupling of $\left(X_t,V_t\right)$ and $(\tilde{X}_t,\tilde{V}_t)$, with initial distribution $(\left(X_0,V_0\right),(\tilde{X}_0,\tilde{V}_0))\sim\Gamma$, previously introduced. 
Using Lemma~\ref{lemme-maj-rho} and Lemma~\ref{maj_K}, by taking the expectation in \eqref{maj-rho} at stopping times $\tau_n$ increasingly converging to $t$ , we have by Fatou's lemma for $n\rightarrow\infty$, $\forall 
\xi>0, \forall t\geq0$,
\begin{align}
 e^{ct}\mathbb{E}\rho_t\leq \mathbb{E}\rho_0+\left(1+\alpha\right)\xi \int_0^te^{cs}\mathbb{E}\left(G_s\right)ds\label{maj-esp-avant-xi}
+L_W\mathcal{C}_K \int_0^te^{cs}\mathbb{E}\rho_sds+L_W\mathcal{C}^0_K \int_0^te^{(c-\gamma  )s}\mathbb{E}\rho_sds.
\end{align}
Moreover, using Lemma~\ref{majMoment} and the fact $\gamma >c$, for all $t\geqslant 0$,
\begin{align*}
\mathbb{E}\left(G_t\right)\leq& \left(1+\epsilon\mathcal{C}^0_H+\epsilon\mathcal{C}^0_{\tilde{H}}\right),\qquad
\int_0^te^{(c-\gamma )s}\mathbb{E}\rho_s ds \leq \frac{f(R_1)\left(1+\epsilon\mathcal{C}^0_H+\epsilon\mathcal{C}^0_{\tilde{H}}\right)}{\gamma -c},\\
\int_0^te^{cs}ds=&\frac{c}{c-L_W\mathcal{C}_K}\frac{e^{ct}-1}{c}-\frac{L_W\mathcal{C}_K}{c-L_W\mathcal{C}_K}\int_0^te^{cs}ds,
\end{align*}
we get
\begin{align*}
 e^{ct}&\left(\mathbb{E}\rho_t-\frac{\left(1+\alpha\right)\xi}{c-L_W\mathcal{C}_K}\left(1+\epsilon\mathcal{C}^0_H+\epsilon\mathcal{C}^0_{\tilde{H}}\right)\right)\\
 \leq& \mathbb{E}\rho_0-\frac{\left(1+\alpha\right)\xi}{c-L_W\mathcal{C}_K}\left(1+\epsilon\mathcal{C}^0_H+\epsilon\mathcal{C}^0_{\tilde{H}}\right)+ L_W\mathcal{C}^0_K\frac{f(R_1)\left(1+\epsilon\mathcal{C}^0_H+\epsilon\mathcal{C}^0_{\tilde{H}}\right)}{\gamma -c}\\
 &+L_W\mathcal{C}_K \int_0^te^{cs}\left(\mathbb{E}\rho_s-\frac{\left(1+\alpha\right)\xi}{c-L_W\mathcal{C}_K}\left(1+\epsilon\mathcal{C}^0_H+\epsilon\mathcal{C}^0_{\tilde{H}}\right)\right)ds.
\end{align*}
Gronwall's lemma yields, for all $t\geq 0$
\begin{equation*}
e^{ct}\left(\mathbb{E}\left(\rho_t\right)-\frac{\left(1+\alpha\right)\xi}{c-L_W\mathcal{C}_K}\left(1+\epsilon\mathcal{C}^0_H+\epsilon\mathcal{C}^0_{\tilde{H}}\right)\right)\leq  \left( \mathbb{E}\left(\rho_0\right)+L_W\mathcal{C}^0_K\frac{f(R_1)\left(1+\epsilon\mathcal{C}^0_H+\epsilon\mathcal{C}^0_{\tilde{H}}\right)}{\gamma -c}\right) e^{L_W \mathcal{C}_K t}.
\end{equation*}
Since $\mathcal{W}_\rho\left(\mu_t,\nu_t\right)\leq\mathbb{E}\left(\rho_t\right)$, we have thus obtained for all $t\geq 0$
\begin{equation*}
\mathcal{W}_\rho\left(\nu^1_t,\nu^2_t\right)\leq \frac{\left(1+\alpha\right)\xi}{c-L_W\mathcal{C}_K}\left(1+\epsilon\mathcal{C}^0_H+\epsilon\mathcal{C}^0_{\tilde{H}}\right)+ \left( \mathbb{E}\left(\rho_0\right)+L_W \mathcal{C}^0_K\frac{f(R_1) \left(1+\epsilon\mathcal{C}^0_H+\epsilon\mathcal{C}^0_{\tilde{H}}\right)}{\gamma -c}\right) e^{(L_W \mathcal{C}_K-c) t}
\end{equation*}
Taking the infimum over all couplings $\Gamma$ of the initial conditions and using the fact that the left hand side does not depend on $\xi$, so that we may take $\xi=0$, we get finally that for all $t\geqslant 0$,
\begin{equation}\label{contrac_rho}
\mathcal{W}_\rho\left(\nu^1_t,\nu^2_t\right)\leq  \left( \mathcal{W}_\rho\left(\nu^1_0,\nu^2_0\right)+ L_W \mathcal{C}^0_K\frac{f(R_1) \left(1+\epsilon\mathcal{C}^0_H+\epsilon\mathcal{C}^0_{\tilde{H}}\right)}{\gamma -c}
\right) e^{(L_W \mathcal{C}_K-c) t},
\end{equation}
and since, by Lemma~\ref{rho_1_2}, $\mathcal{C}_1\mathcal{W}_\rho\left(\nu^1_t,\nu^2_t\right)\geq\mathcal{W}_{1}\left(\nu^1_t,\nu^2_t\right)$ and $\mathcal{C}_2\mathcal{W}_\rho\left(\nu^1_t,\nu^2_t\right)\geq\mathcal{W}_{2}^2\left(\nu^1_t,\nu^2_t\right)$, 
\begin{align*}
\mathcal{W}_{1}\left(\nu^1_t,\nu^2_t\right)\leq& e^{-(c-L_W\mathcal{C}_K)t}C^1_{\nu^1_0,\nu^2_0},\\
\mathcal{W}_{2}^2\left(\nu^1_t,\nu^2_t\right)\leq& e^{-(c-L_W\mathcal{C}_K)t}C^2_{\nu^1_0,\nu^2_0}.
\end{align*}
Then, for all $W$ such that $L_W<c/\mathcal{C}_K$, there will be contraction at rate $\tau:=c-L_W\mathcal{C}_K>0$.  So, it only remains for $L_W$ to satisfy
\begin{equation}\label{maj_L_W}
L_W\leq\frac{c}{\mathcal{C}_1\left(1+\frac{2\epsilon B}{\gamma }\right)+\frac{2\epsilon B}{\gamma\lambda }\left(6+8\lambda\right)},
\end{equation}
with
\begin{align*}
\mathcal{C}_1=\max\left(\frac{2}{\alpha},1\right)\max\left(\frac{4\left(\left(1+\alpha\right)^2+\alpha^2\right)}{\epsilon\min\left(\frac{2}{3}\lambda,6\right)f\left(1\right)},\frac{1}{\phi\left(R_1\right)g\left(R_1\right)}\right).
\end{align*}
\begin{remark}
We draw the reader's attention to the fact that Theorem~\ref{contraction_rho} is then a consequence of everything we have done so far : if we have an upper bound on $\mathbb{E}H\left(X_0,V_0\right)+\mathbb{E}H(\tilde{X}_0,\tilde{V}_0)$, the constant $\mathcal{C}_K^0$ in Lemma~\ref{maj_K} can be chosen equal to 0 provided we modify $\mathcal{C}_K$.
\end{remark}
Let us now show that there is existence and uniqueness of a stationary measure. Let $\mathcal{C}^0>\frac{B}{\gamma}$ and $\mu_t$ a solution of \eqref{Liouville} such that $\mathbb{E}_{\mu_0}H\leq\mathcal{C}^0$. Using \eqref{gronwall_H}, for all $t\geq0$, $\mathbb{E}_{\mu_t}H\leq\mathcal{C}^0$. Thanks to Theorem~\ref{contraction_rho}, for $L_W$ sufficiently small, there is $\tau>0$ such that for all $t\geq s\geq0$
\begin{align*}
\mathcal{W}_{\rho}\left(\mu_t,\mu_s\right)\leq e^{-\tau s}\mathcal{W}_{\rho}\left(\mu_{t-s},\mu_0\right)\leq f(R_1)\left(1+2\epsilon\mathcal{C}^0 \right)e^{-\tau s},
\end{align*}
and thus
\begin{align*}
\mathcal{W}_1\left(\mu_t,\mu_s\right)\leq \mathcal{C}_1f(R_1)\left(1+2\epsilon\mathcal{C}^0 \right)e^{-\tau s}.
\end{align*}
The space of probability measure with first moments, equipped with the $\mathcal{W}_1$ distance, being a complete metric space (see for instance \cite{bolley_separability}), and $\mu_t$ being a Cauchy sequence, there exists $\mu_\infty$ such that $${\mathcal{W}_1\left(\mu_t,\mu_\infty\right)\rightarrow 0} \text{ as }t\rightarrow \infty,$$and $\mu_\infty$ stationary. Theorem~\ref{thm_1} then ensures uniqueness and convergence towards this stationary measure.

%
%Section Systeme de particules
%
%
%

\section{Proof of Theorem~\ref{thm_2}}\label{Preuve_2}

In this section, we show how we obtain similar results for the convergence of the particle system to the non-linear kinetic Langevin diffusion using the same tools. We start by introducing the coupling, the new Lyapunov function, we give a new definition for the various quantities we consider, and then prove contraction of the coupling semimetric.

%
% Subsection Coupling
%

\subsection{Coupling}

We consider the following coupling
\begin{equation*}
\left\{
    \begin{array}{ll}
        d\bar{X}^i_t=\bar{V}^i_tdt\\
        d\bar{V}^i_t=-\bar{V}^i_tdt-\nabla U\left(\bar{X}^i_t\right)dt-\nabla W\ast\bar{\mu}_t\left(\bar{X}^i_t\right)dt+\sqrt{2}\left(rc\left(Z^i_t,W^i_t\right)dB^{rc,i}_t+sc\left(Z^i_t,W^i_t\right)dB^{sc,i}_t\right)\\
        \bar{\mu}_t=\mathcal{L}\left(\bar{X}^i_t\right)\\
        dX^{i,N}_t=V^{i,N}_tdt\\
        dV^{i,N}_t=-V^{i,N}_tdt-\nabla U(X^{i,N}_t)dt-\frac{1}{N}\sum_{j=1}^N\nabla W(X^{i,N}_t-X^{j,N}_t)dt\\
        \hspace{2cm}+\sqrt{2}\left(rc\left(Z^i_t,W^i_t\right)\left(Id-2e^i_te^{i,T}_t\right)dB^{rc,i}_t+sc\left(Z^i_t,W^i_t\right)dB^{sc,i}_t\right),
        \end{array}
\right.
\end{equation*}
with, similarly as before,
\begin{align*}
\begin{array}{cccc}
&Z^i_t=\bar{X}^i_t-X^{i,N}_t,\hspace{0.3cm} &W^i_t=\bar{V}^i_t-V^{i,N}_t, \hspace{0.3cm}&Q^i_t=Z^i_t+W^i_t,\hspace{0.3cm}\\
\end{array}
&e^i_t=\left\{
    \begin{array}{ll}
    \frac{Q^i_t}{|Q^i_t|}\text{ if }Q^i_t\neq0,\\
    0\text{ otherwise}.
    \end{array}
\right.
\end{align*}
Let $\mu^N_t:=\frac{1}{N}\sum_{i=1}^N\delta_{X^{i,N}_t}$ be the empirical distribution of the particle system, with i.i.d initial conditions $X^{i,N}_0\sim\nu_0$. We first notice that the particles are exchangeable.
The generator of the process given by the particle system \eqref{FP_part} 
is, for a function $\phi$ of $(x_1,...,x_N,v_1,...,v_N)$ 
\begin{align*}
\mathcal{L}^N\phi=\sum_{i=1}^N\mathcal{L}^{i,N}\phi,
\end{align*}
with
\begin{align*}
\mathcal{L}^{i,N}\phi=v_i\cdot\nabla_{x_i}\phi-v_i\cdot\nabla_{v_i}\phi-\nabla U\left(x_i\right)\cdot\nabla_{v_i}\phi-\frac{1}{N}\sum_{j=1}^N\nabla W\left(x_i-x_j\right)\cdot\nabla_{v_i}\phi+\Delta_{v_i}\phi.
\end{align*}
We define
\begin{align}
r^i_t=&\alpha|Z^i_t|+|Q^i_t|,\\
\tilde{H}(x,v)=&\int_{0}^{H(x,v)}\exp\left(a\sqrt{u}\right)du\\
\rho_t=&\frac{1}{N}\sum_{i=1}^Nf\left(r^i_t\right)\Big(1+\epsilon \tilde{H}\left(\bar{X}^i_t,\bar{V}^i_t\right)+\epsilon  \tilde{H}(X^{i,N}_t,V^{i,N}_t)\nonumber\\
&\hspace{2.5cm}+\frac{\epsilon}{N}\sum_{j=1}^N  \tilde{H}\left(\bar{X}^j_t,\bar{V}^j_t\right)+\frac{\epsilon}{N}\sum_{j=1}^N  \tilde{H}(X^{j,N}_t,V^{j,N}_t)\Big),\label{def_rho_part}\\
:=&\frac{1}{N}\sum_{i=1}^Nf\left(r^i_t\right)G^i_t.
\end{align}

%
%Exponential of Lyapunov
%

\subsection{A modified Lyapunov function}

Notice how in the expression of $G^i$ above we did not consider the Lyapunov function $H$, but instead $\tilde{H}$. Let us assume there exist $\mathcal{C}_0, a>0$ such that $\mathbb{E}_{\nu_0}\left(\tilde{H}(X,V)^2\right)\leq \left(\mathcal{C}_0\right)^2$ (which is equivalent to the existence of $\tilde{\mathcal{C}^0}, \tilde{a}>0$ such that ${\mathbb{E}_{\nu_0}\left(e^{\tilde{a}\left(|X|+|V|\right)}\right)\leq\tilde{\mathcal{C}^0}}$, as it was stated in Theorem~\ref{thm_2} ).
First, notice
\begin{align*}
\tilde{H}(x,v)=\int_{0}^{H(x,v)}\exp\left(a\sqrt{u}\right)du=\frac{2}{a^2}\exp\left(a\sqrt{H(x,v)}\right)\left(a\sqrt{H(x,v)}-1\right)+\frac{2}{a^2}.
\end{align*}
Direct calculations yield the following technical lemma.
\begin{lemma}
We have, for all $x,v\in\mathbb{R}^d$
\begin{align}
H(x,v)\exp\left(a\sqrt{H(x,v)}\right)\geq \tilde{H}(x,v)\geq& \exp\left(a\sqrt{H(x,v)}\right) -\frac{2}{a^2}\left(\exp\left(\frac{a^2}{2}\right)-1\right)\label{control_tilde_H},\\
\frac{2}{a}\sqrt{H(x,v)}\exp\left(a\sqrt{H(x,v)}\right)\geq \tilde{H}(x,v)\geq&\frac{1}{a}\sqrt{H(x,v)}\exp\left(a\sqrt{H(x,v)}\right)-\frac{1}{a^2}\left(e-2\right)\label{control_tilde_H_2},\\
\tilde{H}(x,v)\geq& H(x,v)\label{control_tilde_H_3}
\end{align}
\end{lemma}

We may calculate, using \eqref{maj_par_H} and \eqref{Lya_H}
\begin{align}
\mathcal{L}_\mu\left(\tilde{H}\right)&=\exp\left(a\sqrt{H}\right)\mathcal{L}_\mu H+\frac{a}{2\sqrt{H}}\exp\left(a\sqrt{H}\right)|\nabla_vH|^2\nonumber\\
&=\exp\left(a\sqrt{H}\right)\mathcal{L}_\mu H+24^2\frac{a}{2\sqrt{H}}\exp\left(a\sqrt{H}\right)\left|\frac{x}{2}+v\right|^2\nonumber\\
&\leq \exp\left(a\sqrt{H}\right)\left(B+L_W\left(6+8\lambda\right)\mathbb{E}_{\mu}\left(|x|\right)^2-\left(\frac{3}{4}\lambda+\lambda^2\right)|x|^2-\gamma H\right)+24a\sqrt{H}\exp\left(a\sqrt{H}\right)\nonumber\\
&\leq \exp\left(a\sqrt{H}\right)\left(B+\frac{288a^2}{\gamma}+L_W\left(6+8\lambda\right)\mathbb{E}_{\mu}\left(|x|\right)^2-\frac{\gamma}{2} H\right),\label{dyn_exp_H}
\end{align}
where for this last inequality we used Young's inequality $24a\sqrt{H}\leq \frac{\gamma}{2}H+288\frac{a^2}{\gamma}$.

Notice that \eqref{dyn_exp_H} ensures that this new Lyapunov function also tends to bring back particle which ventured at infinity, and at an even greater rate. This new rate $H\exp(\sqrt{H})$ however comes at a cost : the initial condition must have a finite exponential moment, and not just a finite second moment as in Section~\ref{Preuve_1}.

First, by \eqref{gronwall_H} and  \eqref{control_tilde_H_3}, $$\mathbb{E}(|\bar{X}^i_t|)^2\leq \frac{1}{\lambda} \mathbb{E}\left(H\left(\bar{X}^i_t,\bar{V}^i_t\right)\right)\leq \frac{1}{\lambda}\left(\frac{B}{\gamma}+\mathbb{E}H\left(\bar{X}^i_0,\bar{V}^i_0\right)\right) \leq\frac{1}{\lambda}\left(\frac{B}{\gamma}+\mathbb{E}\tilde{H}\left(\bar{X}^i_0,\bar{V}^i_0\right)\right)\leq \frac{1}{\lambda}\left(\frac{B}{\gamma}+\mathcal{C}^0\right).$$
Furthermore, the function $h\mapsto \exp\left(a\sqrt{h}\right)\left(\tilde{B}-\frac{\gamma}{4}h\right)$ is bounded from above for $h\geq0$ and $\tilde{B}\in\mathbb{R}$. We therefore obtain from \eqref{dyn_exp_H} the existence of $\tilde{B}$ such that 
\begin{align}
\mathcal{L}_{\bar{\mu}_t^{\otimes N}} \tilde{H}\left(x_i,v_i\right)\leq& \tilde{B}-\frac{\gamma}{4}\left(H\left(x_i,v_i\right)\exp\left(a\sqrt{H\left(x_i,v_i\right)}\right)\right)\label{dyn_tilde_H_non_lin}\\
\frac{d}{dt}\mathbb{E} \tilde{H}\left(\bar{X}^i_t,\bar{V}^i_t\right)\leq& \tilde{B}-\frac{\gamma}{4}\mathbb{E}\left(H\left(\bar{X}^i_t,\bar{V}^i_t\right)\exp\left(a\sqrt{H\left(\bar{X}^i_t,\bar{V}^i_t\right)}\right)\right)\label{dyn_esp_exp_H}\\
\text{and}\ \ \ \frac{d}{dt}\mathbb{E} \tilde{H}\left(\bar{X}^i_t,\bar{V}^i_t\right)\leq&\tilde{B}-\frac{\gamma }{4}\mathbb{E} \tilde{H}\left(\bar{X}^i_t,\bar{V}^i_t\right),\label{Gronwall_exp_H}
\end{align}
where for this last inequality, we used \eqref{control_tilde_H}. While \eqref{dyn_tilde_H_non_lin} and \eqref{dyn_esp_exp_H} will be useful in ensuring a sufficient restoring force, Equation~\eqref{Gronwall_exp_H} give us a uniform in time bound on $\mathbb{E} \tilde{H}\left(\bar{X}^i_t,\bar{V}^i_t\right)$, provided we have an initial bound.
\vspace{0.5cm}

Now, for the system of particle, we have, using \eqref{dyn_exp_H}, $\forall i\in\{1,...,N\},\ \forall x_i,v_i\in\mathbb{R}^d,$
\begin{align*}
\mathcal{L}^{i,N}\tilde{H}\left(x_i,v_i\right)\leq& \exp\left(a\sqrt{H\left(x_i,v_i\right)}\right)\left(B+\frac{288a^2}{\gamma}+L_W\left(6+8\lambda\right)\left(\frac{\sum_{j=1}^N|x_j|}{N}\right)^2-\frac{\gamma}{2}  H\left(x_i,v_i\right)\right).
\end{align*}
Summing over $i\in\{1,..,N\}$, we may calculate
\begin{align}
&L_W\left(6+8\lambda\right)\sum_{j=1}^N\left(\frac{\sum_{j=1}^N|x_j|}{N}\right)^2\sum_{i=1}^N\frac{\exp\left(a\sqrt{H\left(x_i,v_i\right)}\right)}{N}
-\frac{\gamma}{8}\sum_{i=1}^N\frac{H\left(x_i,v_i\right)\exp\left(a\sqrt{H\left(x_i,v_i\right)}\right)}{N}\nonumber\\
\leq&\frac{\gamma}{8}\Big(\sum_{i,j=1}^N\frac{H\left(x_i,v_i\right)}{N}\frac{\exp\left(a\sqrt{H\left(x_j,v_j\right)}\right)}{N} -\sum_{i=1}^N\frac{H\left(x_i,v_i\right)\exp\left(a\sqrt{H\left(x_i,v_i\right)}\right)}{N}\Big)\
\leq\  0\label{annuler_somme_non_line}.
\end{align}
Here, we used \eqref{maj_par_H}, the fact that $\forall x,y\geq0\ \ xe^{a\sqrt{y}}+ye^{a\sqrt{x}}-xe^{a\sqrt{x}}-ye^{a\sqrt{y}}=(e^{a\sqrt{x}}-e^{a\sqrt{y}})(y-x)\leq0$ and assumed $$6\frac{L_W}{\lambda}\left(1+\frac{4}{3}\lambda\right)\leq \frac{\gamma}{8}\ \ \text{ i.e }\ \ L_W\leq\frac{\gamma\lambda}{16(3+4\lambda)}.$$
Likewise, there is a constant, which for the sake of clarity we will also denote $\tilde{B}$ (as we may take the maximum of the previous constants), such that we get 
\begin{align}
\mathcal{L}^{i,N}\tilde{H}(x_i,v_i)\leq&\tilde{B}+L_W\left(6+8\lambda\right)\left(\frac{\sum_{j=1}^N|x_i|}{N}\right)^2\exp\left(a\sqrt{H\left(x_i,v_i\right)}\right)\nonumber\\
&-\frac{\gamma}{4} H\left(x_i,v_i\right)\exp\left(a\sqrt{H\left(x_i,v_i\right)}\right)\label{dyn_part_tilde_H}\\
\mathcal{L}^N\left(\frac{1}{N}\sum_{i=1}^N\tilde{H}(x_i,v_i)\right)\leq& \tilde{B}-\frac{\gamma}{4}\left(\frac{1}{N}\sum_{i=1}^NH(x_i,v_i)\exp\left(a\sqrt{H\left(x_i,v_i\right)}\right)\right)\label{gen_part_sum}
\end{align}
and
\begin{align}
\mathcal{L}^N&\left(\frac{1}{N}\sum_{i=1}^N\tilde{H}(x_i,v_i)\right)\leq \tilde{B}-\frac{\gamma}{4}\left(\frac{1}{N}\sum_{i=1}^N\tilde{H}(x_i,v_i)\right)\label{gron_part_exp}
\end{align}
Once again, \eqref{dyn_part_tilde_H} and \eqref{gen_part_sum} will be ensure a sufficient restoring force, and \eqref{gron_part_exp} ensures a uniform in time bound on the expectation of $\tilde{H}(X^{i,N}_t,V^{i,N}_t)$, since $\mathbb{E}\left(\frac{1}{N}\sum_{j=1}^N\tilde{H}(X^{j,N}_t,V^{j,N}_t)\right)=\mathbb{E}\left(\tilde{H}(X^{i,N}_t,V^{i,N}_t)\right)$ by exchangeability of the particles.

More precisely, we obtain from \eqref{Gronwall_exp_H} and \eqref{gron_part_exp} the direct corollary

\begin{lemma}\label{control_G}
Provided the initial expectations $\mathbb{E}\left(G^1_0\right)$ and $\mathbb{E}\left(\left(G^1_0\right)^2\right)$  are finite, there are constants  $\mathcal{C}_{G,1}$ and $\mathcal{C}_{G,2}$, depending on initial conditions, such that for all $t\geq0$, for all $N\geq0$, and all $i$
\begin{align*}
\mathbb{E}\left(G^i_t\right)\leq\mathcal{C}_{G,1}
\ \ \text{ 
and
 }\ \ 
\mathbb{E}\left(\left(G^i_t\right)^2\right)\leq\mathcal{C}_{G,2}.
\end{align*}
\end{lemma}

Finally, since $\tilde{H}(x,v)\geq H(x,v)$, Lemma~\ref{rho_1_2} still holds for our new semimetric.

%
%Section New parameters
%

\subsection{New parameters}

For the sake of completeness, and since this is similar to Section~\ref{subsection_cstes}, we quickly give some explicit parameters that satisfy the various conditions arising from calculation. T hese parameters are far from optimal, and are just given to show that every constant is explicit. Let $\tilde{B}$ be given by \eqref{dyn_tilde_H_non_lin}-\eqref{Gronwall_exp_H}, and \eqref{dyn_part_tilde_H}-\eqref{gron_part_exp}. Define
$$\alpha=L_U+\frac{\lambda}{4}, \ \  R_0=\sqrt{\frac{160\tilde{B}}{\gamma\min\left(\frac{\lambda}{3},3\right)}}\ \  \text{ and }\ \ R_1=\sqrt{(1+\alpha)^2+\alpha^2}R_0.$$
Recall the definition of $C_{dH,1}$ and $C_{dH,2}$ in \eqref{lemma_dif_H}. Denoting
\begin{align*}
C_{f,1}&=8\left(\left(\frac{96}{a^2}\max\left(1,\frac{1}{2\alpha}\right)+\frac{16\sqrt{3}}{a}C_{dH,1}\right)\left(\exp\left(\frac{a^2}{2}\right)-1\right)+16\sqrt{3}(e-2)C_{dH,2}\right)\\
C_{f,2}&=8\left(24\max\left(1,\frac{1}{2\alpha}\right)+4\sqrt{3}C_{dH,1}a+8\sqrt{3}C_{dH,2}a^2\right)
\end{align*}
we set 
\begin{align*}
c=&\left\{ \frac{1}{12}\min\left(2\sqrt{\frac{L_U+L_W}{2\pi\alpha R_1^2}}, \frac{1}{2}\left(1-\frac{L_U+L_W}{\alpha}\right)\right)\exp\left(-\frac{1}{8}\left(\frac{L_U+L_W}{\alpha}+\alpha+C_{f,1}+C_{f,2}\right)R_1^2\right),\right.\\&\hspace{1cm}\left.\frac{2\tilde{B}}{5}, \frac{\gamma}{800}\right\},
\end{align*}
and $\epsilon=\frac{5c}{2\tilde{B}}$. For $s\geq 0$,
\[\phi\left(s\right)=\exp\left(-\frac{1}{8}\left(\frac{1}{\alpha}\left(L_U+L_W\right)+\alpha+\epsilon C_{f,1}+C_{f,2}\right)s^2\right), \qquad \Phi\left(s\right)=\int_0^{s}\phi\left(u\right)du\]
\[g\left(s\right)=1-\frac{c+2\epsilon\tilde{B}}{2}\int_0^{s}\frac{\Phi\left(u\right)}{\phi\left(u\right)} du \,,\qquad
f\left(s\right)=\int_0^{\min\left(s,R_1\right)}\phi\left(u\right)g\left(u\right)du.\]
This way we satisfy the following conditions
\begin{align*}
c\leq&\frac{\gamma}{160}\left(1-\frac{\gamma}{80\epsilon\tilde{B}+\gamma}\right)\\
\alpha>&L_U+L_W\\
\epsilon\leq&1\\
2c+4\epsilon\tilde{B}\leq&2\left(\int_0^{R_1}\frac{\Phi(u)}{\phi(u)}du\right)^{-1}\\
2c+4\epsilon\tilde{B}\leq&\frac{1}{2}\left(1-\frac{L_U+L_W}{\alpha}\right)\inf_{r\in]0,R_1]}\frac{r\phi(r)}{\Phi(r)}\\
\forall s\geq0, 0=&4\phi'(s)+\left(\frac{1}{\alpha}\left(L_U+L_W\right)+\alpha+\epsilon C_{f,1}+C_{f,2}\right)s\phi(s)
\end{align*}

%
%Section Proof of propagation of chaos
%

\subsection{Convergence}

The goal of the section is to prove the following result
\begin{theorem}\label{thm_2_int}
Let $U\in\mathcal{C}^1\left(\mathbb{R}^{d}\right)$ satisfy Assumption~\ref{HypU} and Assumption~\ref{HypU_lip}. For all $W\in\mathcal{C}^1\left(\mathbb{R}^{d}\right)$ satisfying Assumption~\ref{HypW} with
\begin{align}
L_W\leq\min\left(\frac{\gamma\lambda}{16(3+4\lambda)},\frac{c}{\mathcal{C}_1},\frac{\gamma}{64\mathcal{C}_z},\frac{\gamma a}{256\mathcal{C}_z\epsilon }\right),\label{cond_c_0}
\end{align}
and for all probability measures $\bar{\nu}_0$ on $\mathbb{R}^{2d}$ such that $\mathbb{E}_{\bar{\nu}_0}\tilde{H}^2(X,V)\leq (\mathcal{C}^0)^2$, for all $N$, $ \xi>0$, and $t\geq0, $
\begin{align*}
 e^{ct}\mathbb{E}\left(\rho_t\right)\leq&\mathbb{E}\left(\rho_0\right)+\xi\left(1+\alpha\right)\mathcal{C}_{G,1}\int_{0}^te^{cs}ds+L_W\frac{\mathcal{C}^0\mathcal{C}_{G,2}^{1/2}}{\lambda}\sqrt{\frac{8}{N}}\int_{0}^te^{cs}ds.
\end{align*}
\end{theorem}

%
%Section Proof of contraction
%

\subsubsection{Proof of Theorem~\ref{thm_2} using Theorem~\ref{thm_2_int}}

We first show how Theorem~\ref{thm_2} is a consequence of Theorem~\ref{thm_2_int}.
Let $\Gamma$ be a coupling of $\nu_0^{\otimes N}$ and $\bar{\nu}_0^{\otimes N}$, such that $\mathbb{E}\rho_0<\infty$. We consider the coupling previously introduced. 
For clarity, let us denote
\begin{align*}
\frac{A}{\sqrt{N}}=L_W\frac{\mathcal{C}^0\mathcal{C}_{G,2}^{1/2}}{\lambda}\sqrt{\frac{8}{N}},\ \ \ &B=(1+\alpha)\mathcal{C}_{G,1},
\end{align*}
i.e
\begin{align*}
e^{ct}\mathbb{E}\left(\rho_t\right)\leq&\mathbb{E}\left(\rho_0\right)+\xi 
B\int_{0}^te^{cs}ds+\frac{A}{\sqrt{N}}\int_{0}^te^{cs}ds.
\end{align*}
Let us consider
\begin{align*}
u(t)=e^{ct}&\left(\mathbb{E}\left(\rho_t\right)-\frac{A}{c}\frac{1}{\sqrt{N}}-\xi\frac{B}{c}\right)
\end{align*}
Then $u(t)\leq u(0)$ i.e
\begin{align*}
\mathbb{E}\left(\rho_t\right)\leq&\mathbb{E}\left(\rho_0\right)e^{-ct}+\frac{A}{c}\frac{1}{\sqrt{N}}\left(1-e^{-ct}\right)+\xi\frac{B}{c}\left(1-e^{-ct}\right).
\end{align*}
We thus obtain the desired result, by taking the limit as $\xi\rightarrow0$ uniformly in time, and by using the exchangeability of the particles to have $\mathbb{E}\left(\rho_t\right)=\mathbb{E}\left(\frac{1}{N}\sum_{i=1}^N\rho^i_t\right)=\mathbb{E}\left(\frac{1}{k}\sum_{i=1}^k\rho^i_t\right)$ for all $k\in\mathbb{N}$.

%
%Subsection Evolution of the coupling distance for the particle system
%

\subsubsection{Evolution of the coupling semimetric for the particle system}

We thus need to start by considering the dynamic of $\rho_t$. Like in Lemma~\ref{dyn_Q_t}, we have almost surely for all $t\geq 0$
 \begin{align*}
 d|Q^i_t|=-e^{i,T}_t&\left(\nabla U\left(\bar{X}^i_t\right)-\nabla U(X^{i,N}_t)\right)dt-e^{i,T}_t\left(\nabla W\ast\bar{\mu}_t\left(\bar{X}^i_t\right)-\nabla W\ast\bar{\mu}^N_t(X^{i,N}_t)\right)dt\\
 &+2\sqrt{2}rc\left(Z^i_t,W^i_t\right)e^{i,T}_tdB^{rc,i}_t.
  \end{align*}
 Hence $e^{ct}f\left(r^i_t\right)=f\left(r_0\right)+\hat{A}^i_t+\hat{M}^i_t\text{ with }$
\begin{align*}
d\hat{A}^i_t=&\Big(c f\left(r^i_t\right)+\alpha f'\left(r^i_t\right)\frac{d|Z^i_t|}{dt}- f'\left(r^i_t\right){e^i_t}^T\left(\nabla U\left(X^i_t\right)-\nabla U(X^{i,N}_t)\right)\\
&-f'\left(r^i_t\right){e^i_t}^T\left(\nabla W\ast\mu_t\left(X^i_t\right)-\frac{1}{N}\sum_{j=1}^N\nabla W(X^{i,N}_t-X^{j,N}_t)\right)+4 f''\left(r^i_t\right)rc^2\left(Z^i_t,W^i_t\right)\Big)e^{ct}dt,\\
d\hat{M}^i_t=&e^{ct}2\sqrt{2} f'\left(r^i_t\right)rc\left(Z^i_t,W^i_t\right){e^i_t}^TdB^{rc,i}_t.
\end{align*}
We now consider the evolution of
\begin{align*}
G^i_t=1+\epsilon \tilde{H}\left(\bar{X}^i_t,\bar{V}^i_t\right)+\epsilon  \tilde{H}(X^{i,N}_t,V^{i,N}_t)+\frac{\epsilon}{N}\sum_{j=1}^N \tilde{H}\left(\bar{X}^j_t,\bar{V}^j_t\right)+\frac{\epsilon}{N}\sum_{j=1}^N  \tilde{H}(X^{j,N}_t,V^{j,N}_t).
\end{align*}
Notice how we have added new terms in $G^i_t$. Those additional quantities will help us in dealing with the non linearity, as will be shown later.
\begin{align*}
dG^i_t&=\epsilon\left(\mathcal{L}_{\bar{\mu}_t^{\otimes N}}\tilde{H}\left(\bar{X}^i_t,\bar{V}^i_t\right)+\mathcal{L}^N\tilde{H}(X^{i,N}_t,V^{i,N}_t)\right)dt\\
&+\epsilon\sqrt{2}rc\left(Z^i_t,W^i_t\right)\left(\nabla_v\tilde{H}\left(\bar{X}^i_t,\bar{V}^i_t\right)-\nabla_v\tilde{H}(X^{i,N}_t,V^{i,N}_t)\right) \cdot e^i_t{e^i_t}^TdB^{rc,i}_t\\
&+\epsilon\sqrt{2}rc\left(Z^i_t,W^i_t\right)\left(\nabla_v\tilde{H}\left(\bar{X}^i_t,\bar{V}^i_t\right)+\nabla_v\tilde{H}(X^{i,N}_t,V^{i,N}_t)\right) \cdot \left(Id-e^i_t{e^i_t}^T\right)dB^{rc,i}_t\\
&+\epsilon\sqrt{2}sc\left(Z^i_t,W^i_t\right)\left(\nabla_v\tilde{H}\left(\bar{X}^i_t,\bar{V}^i_t\right)+\nabla_v \tilde{H}(X^{i,N}_t,V^{i,N}_t)\right) \cdot dB^{sc,i}_t\\
&+\frac{\epsilon}{N}\sum_{j=1}^N\left(\mathcal{L}_{\bar{\mu}_t^{\otimes N}}\tilde{H}\left(\bar{X}^j_t,\bar{V}^j_t\right)+\mathcal{L}^N\tilde{H}(X^{j,N}_t,V^{j,N}_t)\right)dt\\
&+\frac{\epsilon\sqrt{2}}{N}\sum_{j=1}^Nrc\left(Z^j_t,W^j_t\right)\left(\nabla_v\tilde{H}\left(\bar{X}^j_t,\bar{V}^j_t\right)-\nabla_v\tilde{H}(X^{j,N}_t,V^{j,N}_t)\right) \cdot e^j_t{e^j_t}^TdB^{rc,j}_t\\
&+\frac{\epsilon\sqrt{2}}{N}\sum_{j=1}^Nrc\left(Z^j_t,W^j_t\right)\left(\nabla_v\tilde{H}\left(\bar{X}^j_t,\bar{V}^j_t\right)+\nabla_v \tilde{H}(X^{j,N}_t,V^{j,N}_t)\right) \cdot \left(Id-e^j_t{e^j_t}^T\right)dB^{rc,j}_t\\
&+\frac{\epsilon\sqrt{2}}{N}\sum_{j=1}^Nsc\left(Z^j_t,W^j_t\right)\left(\nabla_v\tilde{H}\left(\bar{X}^j_t,\bar{V}^j_t\right)+\nabla_v\tilde{H}(X^{j,N}_t,V^{j,N}_t)\right) \cdot dB^{sc,j}_t.
\end{align*}
Therefore
\begin{equation}\label{dyn_rho_prop_1}
e^{ct}\rho^i_t=e^{ct}f\left(r^i_t\right)G^i_t=\rho_0+A^i_t+M^i_t,
\end{equation}
with
\begin{align*}
dA^i_t=&G^i_td\hat{A}^i_t+\epsilon e^{ct}f\left(r^i_t\right)\Big(\mathcal{L}_{\bar{\mu}_t^{\otimes N}}\tilde{H}\left(\bar{X}^i_t,\bar{V}^i_t\right)+\mathcal{L}^N\tilde{H}(X^{i,N}_t,V^{i,N}_t)\\
&+\frac{1}{N}\sum_{j=1}^N\mathcal{L}_{\bar{\mu}_t^{\otimes N}}\tilde{H}\left(\bar{X}^j_t,\bar{V}^j_t\right)+\frac{1}{N}\mathcal{L}^N\sum_{j=1}^N\tilde{H}(X^{j,N}_t,V^{j,N}_t)\Big)dt\\
&+4\epsilon\left(1+\frac{1}{N}\right) e^{ct}f'\left(r^i_t\right)rc^2\left(Z^i_t,W^i_t\right)\left(\nabla_v\tilde{H}\left(\bar{X}^i_t,\bar{V}^i_t\right)-\nabla_v\tilde{H}(X^{i,N}_t,V^{i,N}_t)\right) \cdot e^i_tdt
\end{align*}
and $M^i_t$ is a continuous local martingale. Let us deal with this last line. 
For the sake of conciseness, from now on we denote for all $i$
\begin{align*}
\bar{H}_i:=H\left(\bar{X}^i_t,\bar{V}^i_t\right), \ \ \text{ and }\ \ H^N_i:=H(X^{i,N}_t,V^{i,N}_t)
\end{align*}
We have
\begin{align*}
|\nabla_v&\tilde{H}\left(\bar{X}^i_t,\bar{V}^i_t\right)-\nabla_v \tilde{H}(X^{i,N}_t,V^{i,N}_t)|\\
=&\left|\nabla_v\bar{H}_i\exp\left(a\sqrt{\bar{H}_i}\right)-\nabla_vH^N_i\exp\left(a\sqrt{H^N_i}\right)\right|\\
\leq& \left|12\bar{X}^i_t+24\bar{V}^i_t-12X^{i,N}_t-24V^{i,N}_t\right|\left(\exp\left(a\sqrt{\bar{H}_i}\right)+\exp\left(a\sqrt{H^N_i}\right)\right)\\
&+a\left|12\bar{X}^i_t+24\bar{V}^i_t\right|\left|\sqrt{\bar{H}_i}-\sqrt{H^N_i}\right|\left(\exp\left(a\sqrt{\bar{H}_i}\right)+\exp\left(a\sqrt{H^N_i}\right)\right)\\
\leq&24\max\left(1,\frac{1}{2\alpha}\right)r^i_t\left(\exp\left(a\sqrt{\bar{H}_i}\right)+\exp\left(a\sqrt{H^N_i}\right)\right)\\
&+4a\sqrt{3}\left|\bar{H}_i-H^N_i\right|\left(\exp\left(a\sqrt{\bar{H}_i}\right)+\exp\left(a\sqrt{H^N_i}\right)\right).
\end{align*}
Now, using Lemma~\ref{lemma_dif_H}, we get
\begin{align*}
|\nabla_v&\tilde{H}\left(\bar{X}^i_t,\bar{V}^i_t\right)-\nabla_v \tilde{H}(X^{i,N}_t,V^{i,N}_t)|\\
\leq&\left(24\max\left(1,\frac{1}{2\alpha}\right)+4\sqrt{3}C_{dH,1}a\right)r^i_t\left(\exp\left(a\sqrt{\bar{H}_i}\right)+\exp\left(a\sqrt{H^N_i}\right)\right)\\
&+4\sqrt{3}C_{dH,2}ar^i_t\left(\sqrt{\bar{H}_i}+\sqrt{H^N_i}\right)\left(\exp\left(a\sqrt{\bar{H}_i}\right)+\exp\left(a\sqrt{H^N_i}\right)\right)\\
\leq&\left(24\max\left(1,\frac{1}{2\alpha}\right)+4\sqrt{3}C_{dH,1}a\right)r^i_t\left(\exp\left(a\sqrt{\bar{H}_i}\right)+\exp\left(a\sqrt{H^N_i}\right)\right)\\
&+8\sqrt{3}C_{dH,2}ar^i_t\left(\sqrt{\bar{H}_i}\exp\left(a\sqrt{\bar{H}_i}\right)+\sqrt{H^N_i}\exp\left(a\sqrt{H^N_i}\right)\right).
\end{align*}
Hence why, using \eqref{control_tilde_H} and \eqref{control_tilde_H_2}, we get
\begin{align*}
|\nabla_v&\tilde{H}\left(\bar{X}^i_t,\bar{V}^i_t\right)-\nabla_v \tilde{H}(X^{i,N}_t,V^{i,N}_t)|\\
\leq&\left(24\max\left(1,\frac{1}{2\alpha}\right)+4\sqrt{3}C_{dH,1}a\right)r^i_t\left(\frac{4}{a^2}\left(\exp\left(\frac{a^2}{2}\right)-1\right)+\tilde{H}\left(\bar{X}^i_t,\bar{V}^i_t\right)+\tilde{H}(X^{i,N}_t,V^{i,N}_t)\right)\\
&+8\sqrt{3}C_{dH,2}a^2r^i_t\left(\frac{2}{a^2}(e-2)+\tilde{H}\left(\bar{X}^i_t,\bar{V}^i_t\right)+\tilde{H}(X^{i,N}_t,V^{i,N}_t)\right)
\end{align*}
and thus
\begin{align*}
4\epsilon&\left(1+\frac{1}{N}\right) e^{ct}f'\left(r^i_t\right)rc^2\left(Z^i_t,W^i_t\right)\left(\nabla_v\tilde{H}\left(\bar{X}^i_t,\bar{V}^i_t\right)-\nabla_v\tilde{H}(X^{i,N}_t,V^{i,N}_t)\right) \cdot e^i_tdt\\
\leq&8\epsilon r^i_tf'\left(r^i_t\right)e^{ct}rc^2\left(Z^i_t,W^i_t\right)\left(\left(\frac{96}{a^2}\max\left(1,\frac{1}{2\alpha}\right)+\frac{16\sqrt{3}}{a}C_{dH,1}\right)\left(\exp\left(\frac{a^2}{2}\right)-1\right)+16\sqrt{3}(e-2)C_{dH,2}\right)\\
&+8 r^i_tf'\left(r^i_t\right)e^{ct}rc^2\left(Z^i_t,W^i_t\right)\left(24\max\left(1,\frac{1}{2\alpha}\right)+4\sqrt{3}C_{dH,1}a+8\sqrt{3}C_{dH,2}a^2\right)\left(\epsilon\tilde{H}\left(\bar{X}^i_t,\bar{V}^i_t\right)+\epsilon\tilde{H}(X^{i,N}_t,V^{i,N}_t)\right)\\
\leq& \left(\epsilon C_{f,1}+C_{f,2}\right)r^i_tf'\left(r^i_t\right)rc^2\left(Z^i_t,W^i_t\right)G^i_t.
\end{align*}
Then we use
\begin{align*}
&\left|\nabla W\ast\bar{\mu}_t\left(\bar{X}^i_t\right)-\frac{1}{N}\sum_{j=1}^N\nabla W(X^{i,N}_t-X^{j,N}_t)\right|\\
&\leq\left|\nabla W\ast\bar{\mu}_t\left(\bar{X}^i_t\right)-\frac{1}{N}\sum_{j=1}^N\nabla W\left(\bar{X}^i_t-\bar{X}^j_t\right)\right|+\left|\frac{1}{N}\sum_{j=1}^N\nabla W\left(\bar{X}^i_t-\bar{X}^j_t\right)-\frac{1}{N}\sum_{j=1}^N\nabla W(X^{i,N}_t-X^{j,N}_t)\right|,\\
&\leq\left|\nabla W\ast\bar{\mu}_t\left(\bar{X}^i_t\right)-\frac{1}{N}\sum_{j=1}^N\nabla W\left(\bar{X}^i_t-\bar{X}^j_t\right)\right|+\frac{1}{N}\sum_{j=1}^N\left|\nabla W\left(\bar{X}^i_t-\bar{X}^j_t\right)-\nabla W(X^{i,N}_t-X^{j,N}_t)\right|,\\
&\leq\left|\nabla W\ast\bar{\mu}_t\left(\bar{X}^i_t\right)-\frac{1}{N}\sum_{j=1}^N\nabla W\left(\bar{X}^i_t-\bar{X}^j_t\right)\right|+\frac{L_W}{N}\sum_{j=1}^N\left(\left|\bar{X}^i_t-X^{i,N}_t\right|+\left|\bar{X}^j_t-X^{j,N}_t\right|\right).
\end{align*} 
Thus
\begin{align*}
&\left|\nabla W\ast\bar{\mu}_t\left(\bar{X}^i_t\right)-\frac{1}{N}\sum_{j=1}^N\nabla W(X^{i,N}_t-X^{j,N}_t)\right|\\
&\hspace{3cm}\leq\left|\nabla W\ast\bar{\mu}_t\left(\bar{X}^i_t\right)-\frac{1}{N}\sum_{j=1}^N\nabla W\left(\bar{X}^i_t-\bar{X}^j_t\right)\right|+ L_W|Z^i_t|+L_W\frac{\sum_{j=1}^N|Z^j_t|}{N}.
\end{align*} 
And finally we use \eqref{dyn_tilde_H_non_lin}, \eqref{dyn_part_tilde_H} and \eqref{gen_part_sum} to have
\begin{align*}
\mathcal{L}_{\bar{\mu}_t^{\otimes N}}\tilde{H}&\left(\bar{X}^i_t,\bar{V}^i_t\right)+\mathcal{L}^N\tilde{H}(X^{i,N}_t,V^{i,N}_t)+\frac{1}{N}\sum_{j=1}^N\mathcal{L}_{\bar{\mu}_t^{\otimes N}}\tilde{H}\left(\bar{X}^j_t,\bar{V}^j_t\right)+\frac{1}{N}\mathcal{L}^N\sum_{j=1}^N\tilde{H}(X^{j,N}_t,V^{j,N}_t)\\
\leq&4\tilde{B}+L_W\left(6+8\lambda\right)\left(\frac{\sum_{j=1}^N|X^{j,N}_t|}{N}\right)^2\exp\left(a\sqrt{H^N_i}\right)-\frac{\gamma }{4}\bar{H}_i\exp\left(a\sqrt{\bar{H}_i}\right)-\frac{\gamma }{4}  H^N_i\exp\left(a\sqrt{H^N_i}\right)\\
&-\frac{\gamma }{4N}\sum_{j=1}^N\left(\bar{H}_j\exp\left(a\sqrt{\bar{H}_j}\right)+H^N_j\exp\left(a\sqrt{H^N_j}\right)\right).
\end{align*} 
We thus obtain
\begin{equation}\label{dyn_rho_prop_2}
dA^i_t\leq e^{ct} K^i_tdt
\end{equation} 
with
\begin{align}
K^i_t=&f'\left(r^i_t\right)G^i_t\left(\alpha\frac{d|Z^i_t|}{dt}+\left(L_U+L_W\right)|Z^i_t|+\left(\epsilon C_{f,1}+C_{f,2}\right)r^i_trc^2\left(Z^i_t,W^i_t\right)\right)+2cf\left(r^i_t\right)G^i_t \label{K_t_1}\\
 &+4f''\left(r^i_t\right)G^i_trc^2\left(Z^i_t,W^i_t\right)+\left|\nabla W\ast\bar{\mu}_t\left(\bar{X}^i_t\right)-\frac{1}{N}\sum_{j=1}^N\nabla W\left(\bar{X}^i_t-\bar{X}^j_t\right)\right|f'\left(r^i_t\right)G^i_t\label{K_t_2}\\
 &+\epsilon f\left(r^i_t\right)\left(4\tilde{B}-\frac{\gamma}{16}  \tilde{H}\left(\bar{X}^i_t,\bar{V}^i_t\right)-\frac{\gamma}{16}  \tilde{H}(X^{i,N}_t,V^{i,N}_t) -\frac{\gamma }{16N}\sum_{j=1}^N\tilde{H}\left(\bar{X}^j_t,\bar{V}^j_t\right)\right.\nonumber\\&\left.\qquad\qquad\qquad-\frac{\gamma }{16N}\sum_{j=1}^N\tilde{H}(X^{j,N}_t,V^{j,N}_t)\right)\label{K_t_5}\\
&+L_W\frac{\sum_{j=1}^N|Z^j_t|}{N}f'\left(r^i_t\right)G^i_t-cf\left(r^i_t\right)G^i_t-\epsilon f\left(r^i_t\right)\Big(\frac{\gamma}{16}  \bar{H}_i\exp\left(a\sqrt{\bar{H}_i}\right)+\frac{\gamma}{16}  H^N_i\exp\left(a\sqrt{H^N_i}\right)\nonumber\\
&\hspace{2cm}+ \frac{\gamma }{16N}\sum_{j=1}^N\bar{H}_j\exp\left(a\sqrt{\bar{H}_j}\right)+\frac{\gamma }{16N}\sum_{j=1}^NH^N_j\exp\left(a\sqrt{H^N_j}\right)\Big)\label{K_t_4}\\\
&+\epsilon L_W\left(6+8\lambda\right)f\left(r^i_t\right)\left(\frac{\sum_{j=1}^N|X^{j,N}_t|}{N}\right)^2\exp\left(a\sqrt{H^N_i}\right)\nonumber\\
&\hspace{2cm}- \frac{\gamma\epsilon }{8}f\left(r^i_t\right)\left(H^N_i\exp\left(a\sqrt{H^N_i}\right)+\frac{1}{N}\sum_{j=1}^NH^N_j\exp\left(a\sqrt{H^N_j}\right)\right)\label{K_t_8}.
\end{align}
This formulation of $K^i_t$ might seem cumbersome (and to some degree it is...) but we have actually grouped the various terms based on how we will have them compensate one another. Thus, 
\begin{itemize}
\item lines \eqref{K_t_1} and \eqref{K_t_2} will be managed thanks to the 
construction of the function $f$ like before, with a special care given to the last term of line \eqref{K_t_2}, on which we will use a law of large number,
\item line \eqref{K_t_5} will come into play when considering the "last region of space" introduced previously,
\item line \eqref{K_t_4} will, under some conditions on $L_W$, be nonpositive when summing up all $\left(K^j_t\right)_j$,
\item and finally, line \eqref{K_t_8} will be nonpositive thanks to Lemma~\ref{Lya}, provided $L_W$ is sufficiently small. 
\end{itemize}
This shows an important idea in the construction of the function $\rho$ : we added in $G^i_t$ the empirical mean of $H(X^{i,N}_t,V^{i,N}_t)+H\left(\bar{X}^i_t,\bar{V}^i_t\right)$. This will allow us to tackle the non linearity appearing through $\frac{\sum_{j=1}^N|Z^j_t|}{N}$.

%
%Section Some calculations
%

\subsubsection{Some calculations}

Like previously, we now have to show contraction in all three regions of space. Recall $f'\left(r^i_t\right)\leq1$. The same calculations as before will be used, we only detail here the differences.

\begin{itemize}
\item First, since $\frac{L_W}{\lambda}\left(6+8\lambda\right)\leq\frac{\gamma}{8} $, by using Lemma~\ref{Lya} and since
\begin{align*}
H^N_j\exp\left(a\sqrt{H^N_i}\right)\leq H^N_i\exp\left(a\sqrt{H^N_i}\right)+H^N_j\exp\left(a\sqrt{H^N_j}\right)
\end{align*}
we obtain
 \begin{align*}
\epsilon L_W&\left(6+8\lambda\right)f\left(r^i_t\right)\left(\frac{\sum_{j=1}^N|X^{j,N}_t|}{N}\right)^2\exp\left(a\sqrt{H^N_i}\right)\\
&\hspace{2cm}- \frac{\gamma\epsilon }{8N}f\left(r^i_t\right)\left(NH^N_i\exp\left(a\sqrt{H^N_i}\right)+\sum_{j=1}^NH^N_j\exp\left(a\sqrt{H^N_j}\right)\right)\leq0.
\end{align*} 
This takes care of \eqref{K_t_8}.

\item We have, since $f'\left(r^i_t\right)\leq1$: 
$\frac{1}{N}\sum_{i=1}^N\frac{\sum_{j=1}^N|Z^j_t|}{N}f'\left(r^i_t\right)G^i_t\leq\frac{\sum_{i,j=1}^N|Z^j_t|G^i_t}{N^2}.
$\\ Then, using Lemma~\ref{rho_1_2}
\begin{align*}
\frac{1}{N^2}\sum_{i,j=1}^N|Z^i_t|G^j_t=&\frac{1}{N}\sum_{i=1}^N|Z^i_t|+\frac{2\epsilon}{N^2} \sum_{i,j=1}^N |Z^i_t|\tilde{H}\left(\bar{X}^j_t,\bar{V}^j_t\right)+\frac{2\epsilon}{N^2} \sum_{i,j=1}^N|Z^i_t| \tilde{H}(X^{j,N}_t,V^{j,N}_t)\\
\leq&\frac{\mathcal{C}_1}{N}\sum_{i=1}^N\rho^i_t+\frac{2\mathcal{C}_z\epsilon}{N^2} \sum_{i,j=1}^N f(r^i_t)\left(\tilde{H}\left(\bar{X}^j_t,\bar{V}^j_t\right)+ \tilde{H}(X^{j,N}_t,V^{j,N}_t)\right)\\
&+\frac{2\mathcal{C}_z\epsilon^2}{N^2}\frac{2}{a} \sum_{i,j=1}^N f(r^i_t)\left(\sqrt{\bar{H}_i}+\sqrt{H^N_i}\right)\left(\sqrt{\bar{H}_j}\exp\left(a\sqrt{\bar{H}_j}\right)+\sqrt{H^N_j}\exp\left(a\sqrt{H^N_j}\right)\right).
\end{align*}
First, using \eqref{control_tilde_H}
\begin{align*}
\frac{2\mathcal{C}_z\epsilon}{N^2} \sum_{i,j=1}^N f(r^i_t)\left(\tilde{H}\left(\bar{X}^j_t,\bar{V}^j_t\right)+ \tilde{H}(X^{j,N}_t,V^{j,N}_t)\right)\leq \frac{2\mathcal{C}_z\epsilon}{N^2} \sum_{i,j=1}^N f(r^i_t)\left(\bar{H}_j\exp\left(a\sqrt{\bar{H}_j}\right)+H^N_j\exp\left(a\sqrt{H^N_j}\right)\right).
\end{align*}
Since
\begin{align*}
\left(\sqrt{\bar{H}_i}+\sqrt{H^N_i}\right)&\left(\sqrt{\bar{H}_j}\exp\left(a\sqrt{\bar{H}_j}\right)+\sqrt{H^N_j}\exp\left(a\sqrt{H^N_j}\right)\right)\\
\leq&2\left(\bar{H}_i\exp\left(a\sqrt{\bar{H}_i}\right)+\bar{H}_j\exp\left(a\sqrt{\bar{H}_j}\right)+H^N_i\exp\left(a\sqrt{H^N_i}\right)+H^N_j\exp\left(a\sqrt{H^N_j}\right)\right),
\end{align*}
we have
\begin{align*}
\frac{4\mathcal{C}_z\epsilon^2}{aN^2} \sum_{i,j=1}^N& f(r^i_t)\left(\sqrt{\bar{H}_i}+\sqrt{H^N_i}\right)\left(\sqrt{\bar{H}_j}\exp\left(a\sqrt{\bar{H}_j}\right)+\sqrt{H^N_j}\exp\left(a\sqrt{H^N_j}\right)\right)\\
\leq&\frac{8\mathcal{C}_z\epsilon^2}{aN^2} \sum_{i,j=1}^N f(r^i_t)\left(\bar{H}_i\exp\left(a\sqrt{\bar{H}_i}\right)+\bar{H}_j\exp\left(a\sqrt{\bar{H}_j}\right)+H^N_i\exp\left(a\sqrt{H^N_i}\right)+H^N_j\exp\left(a\sqrt{H^N_j}\right)\right).
\end{align*}
This way, since $2\mathcal{C}_zL_W\leq\frac{\gamma}{32}$, $L_W\epsilon\frac{8\mathcal{C}_z}{a}\leq\frac{\gamma}{32}$,  and $L_W\mathcal{C}_1\leq c$, we get
\begin{align*}
&\frac{1}{N}\sum_{i=1}^N\left(L_W\frac{\sum_{j=1}^N|Z^j_t|}{N}f'\left(r^i_t\right)G^i_t-cf\left(r^i_t\right)G^i_t-\epsilon f\left(r^i_t\right)\left(\frac{\gamma}{16}  \bar{H}_i\exp\left(a\sqrt{\bar{H}_i}\right)+\frac{\gamma}{16}  H^N_i\exp\left(a\sqrt{H^N_i}\right)\right.\right.\\
&\left.\left.\hspace{2cm}+ \frac{\gamma }{16N}\sum_{j=1}^N\bar{H}_j\exp\left(a\sqrt{\bar{H}_j}\right)+\frac{\gamma }{16N}\sum_{j=1}^NH^N_j\exp\left(a\sqrt{H^N_j}\right)\right)\right)\leq0
\end{align*}

\item Using Cauchy-Schwarz inequality
\begin{align*}
\mathbb{E}&\left(G^i_t\left|\nabla W\ast\bar{\mu}_t\left(\bar{X}^i_t\right)-\frac{1}{N}\sum_{j=1}^N\nabla W\left(\bar{X}^i_t-\bar{X}^j_t\right)\right|\right)\\
&\leq\mathbb{E}\left({G^i_t}^2\right)^{1/2}\mathbb{E}\left(\left|\nabla W\ast\bar{\mu}_t\left(\bar{X}^i_t\right)-\frac{1}{N}\sum_{j=1}^N\nabla W\left(\bar{X}^i_t-\bar{X}^j_t\right)\right|^2\right)^{1/2},\\
&\leq\mathbb{E}\left({G^i_t}^2\right)^{1/2}\mathbb{E}\left(\mathbb{E}\left(\left|\nabla W\ast\bar{\mu}_t\left(\bar{X}^i_t\right)-\frac{1}{N}\sum_{j=1}^N\nabla W\left(\bar{X}^i_t-\bar{X}^j_t\right)\right|^2\Big|\bar{X}^i_t\right)\right)^{1/2}.
\end{align*} 
Moreover, we notice that given $\bar{X}^i_t$, the random variables $\bar{X}^j_t$ for $j\neq i$ are i.i.d with law $\bar{\mu}_t$. Hence
\begin{align*}
\mathbb{E}\left(\left|\nabla W\ast\bar{\mu}_t\left(\bar{X}^i_t\right)-\frac{1}{N-1}\sum_{j=1}^N\nabla W\left(\bar{X}^i_t-\bar{X}^j_t\right)\right|^2\Big|\bar{X}^i_t\right)&=\frac{1}{N-1}\text{Var}_{\bar{\mu}_t}\left(\nabla W\left(\bar{X}^i_t-\cdot\right)\right),\\
&\leq\frac{2L_W^2}{N-1}\mathbb{E}_{\bar{\mu}_t}\left(|\cdot|^2\right),
\end{align*} 
so
\begin{align*}
&\mathbb{E}\Bigg(\left|\nabla W\ast\bar{\mu}_t\left(\bar{X}^i_t\right)-\frac{1}{N}\sum_{j=1}^N\nabla W\left(\bar{X}^i_t-\bar{X}^j_t\right)\right|^2\Bigg)\\
&\hspace{3cm}\leq\mathbb{E}\left(\left|\nabla W\ast\bar{\mu}_t\left(\bar{X}^i_t\right)-\frac{1}{N-1}\sum_{j=1}^N\nabla W\left(\bar{X}^i_t-\bar{X}^j_t\right)\right|^2\right)\\
&\hspace{3.2cm}+\mathbb{E}\left(\left|\frac{1}{N}\sum_{j=1}^N\nabla W\left(\bar{X}^i_t-\bar{X}^j_t\right)-\frac{1}{N-1}\sum_{j=1}^N\nabla W\left(\bar{X}^i_t-\bar{X}^j_t\right)\right|^2\right),\\
&\hspace{3cm}\leq\frac{2L_W^2}{N-1}\mathbb{E}_{\bar{\mu}_t}\left(|\cdot|^2\right)+\left(\frac{1}{N-1}-\frac{1}{N}\right)^2N\sum_{j=1}^NL_W^2\mathbb{E}\left(|\bar{X}^i_t-\bar{X}^j_t|^2\right),\\
&\hspace{3cm}\leq 2L_W^2\left(\frac{1}{N-1}+\frac{1}{(N-1)^2}\right)\mathbb{E}_{\bar{\mu}_t}\left(|\cdot|^2\right).
\end{align*} 
We may then use $\mathbb{E}_{\bar{\mu}_t}\left(|\cdot|^2\right)\leq\frac{\mathcal{C}^0}{\lambda}$.
\end{itemize}
Thus, by the same exact construction as before,  we can obtain the existence of a function $f$ and a constant $c>0$ such that in all regions of space, for $L_W$ sufficiently small, 
\begin{align*}
\mathbb{E}\left(\frac{1}{N}\sum_iK^i_t\right)\leq& \xi\left(1+\alpha\right)\mathcal{C}_{G,1}+L_W\frac{\mathcal{C}^0\mathcal{C}_{G,2}^{1/2}}{\lambda}\left(\frac{2}{N-1}+\frac{2}{(N-1)^2}\right)^{1/2}.
\end{align*} 
By taking the expectation in the dynamic of $\rho_t$ given by \eqref{dyn_rho_prop_1} and\eqref{dyn_rho_prop_2} at stopping times $\tau_n$ increasingly converging to $t$ , we prove Theorem~\ref{thm_2_int} by using Fatou's lemma for $n\rightarrow\infty$.

%
%Appendix
%

\appendix

%
%Section
%

\section{Various results}

%
%Subsection
%

\subsection{Proof of lemma \ref{minU}}\label{preuve_min_U}

The property only depends on the distance to the the origin, not the direction. We therefore only need to prove it in dimension 1, making sure the 
constant $\tilde{A}$ is independent of the direction. There is $x_0>0$ such that $\frac{\lambda}{2}x_0^2=2A$.
Therefore, for $x\geq0$, using \eqref{HypU-lya}:
\begin{align*}
U'\left(x_0+x\right)\left(x_0+x\right)\geq 2\lambda U\left(x_0+x\right)+\frac{\lambda}{2}\left(x_0+x\right)^2-2A=2\lambda U\left(x_0+x\right)+\frac{\lambda}{2}x^2+\lambda x x_0.
\end{align*}
Then, for $x\geq0$:
\begin{align*}
U\left(x_0+x\right)-U\left(x_0\right)&=\int_0^1U'\left(x_0+tx\right)xdt=\int_0^1U'\left(x_0+tx\right)\left(x_0+tx\right)\frac{x}{x_0+tx}dt\\
&\geq\frac{x}{x_0+x}\int_0^1 2\lambda U\left(x_0+tx\right)+\frac{\lambda}{2}t^2x^2+\lambda t x x_0dt\\
&\geq\frac{x}{x_0+x}\left(\frac{\lambda}{6}x^2+\frac{\lambda}{2} x x_0\right)\hspace{0.5cm}\text{since }U\geq0\\
&=\frac{\lambda}{6}\frac{x^3}{x_0+x}+\frac{\lambda}{2}\frac{x^2x_0}{x_0+x}.
\end{align*}
We thus have for all $x\geq x_0$:
\begin{align*}
U\left(x\right)-U\left(x_0\right)&\geq\frac{\lambda}{6}\frac{\left(x-x_0\right)^3}{x}+\frac{\lambda}{2}\frac{\left(x-x_0\right)^2x_0}{x}\\
&=\frac{\lambda}{6}x^2-\frac{\lambda}{2}xx_0+\frac{\lambda}{2}x_0^2-\frac{\lambda}{6}\frac{x_0^3}{x}+\frac{\lambda}{2}xx_0-\lambda x_0^2+\frac{\lambda}{2}\frac{x_0^3}{x}\\
&=\frac{\lambda}{6}x^2-\frac{\lambda}{2}x_0^2+\frac{\lambda}{3}\frac{x_0^3}{x}.
\end{align*}
However, $-\frac{\lambda}{2}x_0^2+\frac{\lambda}{3}\frac{x_0^3}{x}\geq-\frac{\lambda}{2}x_0^2=-2A$ for $x\geq x_0$. We therefore have the desired result for $x\geq x_0$. The same reasoning gives us the result for $x\leq -x_0$. 

Hence, if $|x|\geq|x_0|=\sqrt{\frac{4A}{\lambda}}$, $U\left(x\right)-U\left(x_0\right)\geq \frac{\lambda}{6}x^2-2A$. We then use the fact that $U\left(x\right)$ is continuous on the sphere of center 0 and radius $\sqrt{\frac{4A}{\lambda}}$, hence bounded on this set, to give a lower bound on $U\left(x_0\right)$ independent of the direction. Finally, for $|x|\in[-x_0,x_0]$, the function $x\mapsto U\left(x\right)-\frac{\lambda}{6}x^2$ 
is continuous, therefore bounded. 

\subsection{Proof of lemma \ref{majNablaW}}\label{preuve_majNablaW}

We have
\begin{align*}
\nabla W\ast\mu\left(x\right)-\nabla W\ast\nu\left(\tilde{x}\right)=\nabla W\ast\mu\left(x\right)-\nabla W\ast\mu\left(\tilde{x}\right)+\nabla W\ast\mu\left(\tilde{x}\right)-\nabla W\ast\nu\left(\tilde{x}\right)
\end{align*}
Let $(X,\tilde{X})$ be a coupling of $\mu$ and $\nu$. Then
\begin{align*}
|\nabla W\ast\mu_t\left(x\right)-\nabla W\ast\tilde{\mu}_t\left(\tilde{x}\right)|=&\left|\mathbb{E}\left(\nabla W(x-X)-\nabla W(\tilde{x}-\tilde{X})\right)\right|\\
\leq&L_W\mathbb{E}\left(|x-X-\tilde{x}+\tilde{X}|\right)\\
\leq&L_W\mathbb{E}\left(|x-\tilde{x}|+|X-\tilde{X}|\right)\\
\end{align*}
This being true for all coupling, we obtain the desired result.

%
%Section
%

\subsection{Proof of Lemma~\ref{Lya}}\label{preuve_lya}

\begin{remark}
With $\gamma$ given by \eqref{eq:def_gamma_B}, we have $\gamma \leq\frac{1}{2}$.
\end{remark}
We have
\begin{align*}
\mathcal{L}_\mu H\left(x,v\right)&=v\cdot\nabla_xH\left(x,v\right)-v\cdot\nabla_vH\left(x,v\right)-\nabla U\left(x\right)\cdot\nabla_vH\left(x,v\right)-\nabla W\ast\mu\left(x\right)\cdot\nabla_vH\left(x,v\right)\\
&\hspace{10cm}+\Delta_vH\left(x,v\right)\\
&=v\cdot\left(24\nabla U\left(x\right)+12\left(1-\gamma \right)x+2\lambda x+12v\right)-v\cdot\left(12x+24v\right)\\
&\hspace{1cm}-\nabla U\left(x\right)\cdot\left(12x+24v\right)-\nabla W\ast\mu\left(x\right)\cdot\left(12x+24v\right)+24d\\
&=24d-12\nabla U\left(x\right)\cdot x+x\cdot v\left(12\left(1-\gamma \right)+2\lambda-12\right)-\nabla W\ast\mu\left(x\right)\cdot\left(12x+24v\right)-12|v|^2,
\end{align*}
with
\begin{align*}
-\gamma H\left(x,v\right)=&-24\gamma U\left(x\right)-6\gamma \left(1-\gamma \right)|x|^2-\gamma\lambda |x|^2-12\gamma  x\cdot v-12\gamma |v|^2\\
-12\nabla U\left(x\right)\cdot x\leq&-24\lambda U\left(x\right)-6\lambda|x|^2+24A\\
-\nabla W\ast\mu\left(x\right)\cdot\left(12x+24v\right)\leq& \left(L_W|x|+L_W\mathbb{E}_{\mu}\left(|\cdot|\right)\right)\left(12|x|+24|v|\right)\\
\leq& 12L_W|x|^2+24L_W|x||v|+L_W\mathbb{E}_{\mu}\left(|\cdot|\right)\Big(6\frac{|x|^2}{a_x\mathbb{E}_{\mu}\left(|\cdot|\right)}+6a_x\mathbb{E}_{\mu}\left(|\cdot|\right)\\
&+12\frac{|v|^2}{a_v\mathbb{E}_{\mu}\left(|\cdot|\right)}+12a_v\mathbb{E}_{\mu}\left(|\cdot|\right)\Big),
\end{align*}
where this last inequality holds for any $a_x, a_v>0$. Therefore
\begin{align*}
\mathcal{L}_ \mu H\left(x,v\right)\leq& 24A+24d+6L_W\mathbb{E}_{\mu}\left(|\cdot|\right)^2\left(a_x+2a_v\right)-\gamma  H\left(x,v\right)+24\gamma 
 U\left(x\right)+6\gamma \left(1-\gamma \right)|x|^2\\
&+\gamma\lambda |x|^2+12\gamma  x\cdot v+12\gamma |v|^2-24\lambda U\left(x\right)-6\lambda|x|^2+12L_W|x|^2+24L_W|x||v|\\
&+\frac{6L_W}{a_x}|x|^2+\frac{12L_W}{a_v}|v|^2+x\cdot v\left(12\left(1-\gamma \right)+2\lambda-12\right)-12|v|^2,
\end{align*}
and then
\begin{align*}
\mathcal{L}_\mu H\left(x,v\right)\leq& 24A+24d+6L_W\mathbb{E}_{\mu}\left(|\cdot|\right)^2\left(a_x+2a_v\right)-\gamma  H\left(x,v\right)+24 U\left(x\right)\left(\gamma-\lambda \right)\\
&+|x||v|\left(|12\gamma +12\left(1-\gamma \right)+2\lambda-12|+24L_W\right)\\
&+|x|^2\left(6\gamma \left(1-\gamma \right)+\gamma\lambda -6\lambda+12L_W+\frac{6L_W}{a_x}\right)\\
&+|v|^2\left(12\gamma -12+\frac{12L_W}{a_v}\right).
\end{align*}
We now use $|x||v|\leq \frac{\lambda}{3}|x|^2+\frac{3}{4\lambda}|v|^2$, and $|12\gamma\lambda+12\left(1-\gamma\lambda\right)+2\lambda-12|=2\lambda$.

We have $\left(\gamma-\lambda \right)<0$. Hence $24 U\left(x\right)\left(\gamma-\lambda \right)\leq4\lambda\left(\gamma-\lambda \right)|x|^2-24\left(\gamma-\lambda\right)\tilde{A}$ using Lemma~\ref{minU}. Then
\begin{align*}
\mathcal{L}_\mu H\left(x,v\right)&\leq 24A-24 \left(\gamma-\lambda\right)\tilde{A}+24d+6L_W\mathbb{E}_{\mu}\left(|\cdot|\right)^2\left(a_x+2a_v\right)-\gamma  H\left(x,v\right)\\
&+|x|^2\left(4\lambda \left(\gamma-\lambda\right)+6\gamma \left(1-\gamma \right)+\gamma\lambda -6\lambda+\frac{6L_W}{a_x}+12L_W+\frac{2\lambda^2}{3}+\frac{24L_W\lambda}{3}\right)\\
&+|v|^2\left(12\gamma -12+\frac{12L_W}{a_v}+\frac{3}{2}+\frac{3}{4\lambda}24L_W\right).
\end{align*}
We now consider each term individually.
\begin{description}
\item[Coefficient of $|x|^2$.] We have, using $0<\gamma <1$ and $L_W\leq\frac{\lambda}{8}$
%+\frac{6L_W}{a_x}
\begin{align*}
4\lambda \left(\gamma-\lambda\right)+&6\gamma \left(1-\gamma \right)+\gamma\lambda -6\lambda+\frac{2\lambda^2}{3}+\frac{24L_W\lambda}{3}+12L_W\\
&\leq \gamma\left(5\lambda +6 \right)-\left(4\lambda^2+6\lambda-\frac{2\lambda^2}{3}-\lambda^2-\frac{3\lambda}{2}\right).
\end{align*}
Therefore, it is sufficient that
\begin{align*}
\gamma\leq\lambda\frac{\frac{7}{3}\lambda +\frac{9}{2} }{5\lambda+6}.
\end{align*}
We check this holds for $\gamma=\frac{\lambda}{2\lambda+2}$. Then
\begin{align*}
4\lambda \left(\gamma-\lambda\right)+&6\gamma \left(1-\gamma \right)+\gamma\lambda -6\lambda+\frac{2\lambda^2}{3}+\frac{24L_W\lambda}{3}+12L_W+\frac{6L_W}{a_x}\\
&\leq\left(5\lambda +6 \right)\left(\gamma-\frac{\frac{7}{3}\lambda^2+\frac{9}{2}\lambda}{5\lambda+6}\right)+\frac{3\lambda}{4a_x}.
\end{align*}
We therefore choose 
\begin{align*}
\frac{3\lambda}{4a_x}\leq-\left(5\lambda +6 \right)\left(\gamma-\frac{\frac{7}{3}\lambda+\frac{9}{2}\lambda}{5\lambda+6}\right)=\frac{7}{3}\lambda^2+\frac{9}{2}\lambda-\frac{5\lambda^2+6\lambda}{2\lambda+2}.
\end{align*}
It is, for that, sufficient to take
\begin{align*}
\frac{3\lambda}{4a_x}=\frac{3}{4}\lambda,\hspace{0.3cm}\text{i.e } \hspace{0.3cm}a_x=1.
\end{align*}
Furthermore
\begin{align*}
6\lambda \left(\gamma-\lambda\right)&+6\gamma \left(1-\gamma \right)+\gamma\lambda -6\lambda+\frac{2\lambda^2}{3}+\frac{24L_W\lambda}{3}+12L_W+\frac{6L_W}{a_x}\\
&\leq\frac{5\lambda^2+6\lambda}{2\lambda+2}-\frac{7}{3}\lambda^2-\frac{9}{2}\lambda+\frac{3}{4}\lambda=\frac{5\lambda^2+6\lambda}{2\lambda+2}-\frac{7}{3}\lambda^2-\frac{15}{4}\lambda.
\end{align*}
We then observe
\begin{align*}
6\lambda \left(\gamma-\lambda\right)+6\gamma \left(1-\gamma \right)+\gamma\lambda -6\lambda+\frac{6L_W}{a_x}+\frac{2\lambda^2}{3}+\frac{24L_W\lambda}{3}+12L_W\leq -\lambda^2-\frac{3}{4}\lambda.
\end{align*}
And finally
\begin{align*}
\forall \lambda >0,\  \forall x,\hspace{0.5cm}|x|^2\left(6\lambda \left(\gamma-\lambda\right)+6\gamma \left(1-\gamma \right)+\gamma\lambda -6\lambda+\frac{6L_W}{a_x}+\frac{2\lambda^2}{3}+\frac{48L_W\lambda}{3}+12L_W\right)\\
\leq-\lambda^2|x|^2-\frac{3}{4}\lambda|x|^2
\end{align*}

\item[Coefficient of $|v|^2$.] We have, using $0<\gamma \leq\frac{1}{2}$ and $L_W\leq\lambda/8$
%+\frac{12L_W}{a_v}
\begin{align*}
12\gamma -12+\frac{3}{2}+\frac{3}{4\lambda}24L_W\leq-6+\frac{3}{2}+\frac{18}{\lambda}\cdot\frac{\lambda}{8}=-6+\frac{3}{2}+\frac{9}{4}=-\frac{9}{4}.
\end{align*}
We then choose
\begin{align*}
\frac{12\lambda}{8a_v}=\frac{9}{4},\hspace{0.3cm}\text{ i.e } \hspace{0.3cm}a_v=\frac{2}{3}\lambda.
\end{align*}
Therefore
\begin{align*}
\forall \lambda >0,\  \forall v,\hspace{0.5cm}|v|^2\left(12\gamma -12+\frac{3}{2}+\frac{3}{4\lambda}24L_W+\frac{12L_W}{a_v}\right)\leq0.
\end{align*}
\end{description}
We thus obtain
\begin{align*}
\mathcal{L}_\mu H\left(x,v\right)\leq 24\left(A- \left(\gamma-\lambda\right)\tilde{A}+d\right)+6L_W\mathbb{E}_{\mu}\left(|\cdot|\right)^2\left(1+\frac{4}{3}\lambda\right)-\lambda^2|x|^2-\frac{3}{4}\lambda|x|^2-\gamma  H\left(x,v\right),
\end{align*}
i.e
\begin{equation}\label{majTemp}
\mathcal{L}_\mu H\left(x,v\right)\leq 24\left(A- \left(\gamma-\lambda\right)\tilde{A}+d\right)+\mathbb{E}_{\mu}\left(|\cdot|\right)^2\left(\frac{3}{4}\lambda+\lambda^2\right)-\lambda^2|x|^2-\frac{3}{4}\lambda|x|^2-\gamma  H\left(x,v\right).
\end{equation}

\subsection{Proof of Lemma~\ref{Prop-H}}\label{preuve_prop_de_H}

Using $1-\gamma \geq\frac{1}{2}$, we get
\begin{align*}
H\left(x,v\right)&\geq24U(x)+\left(3+\lambda\right)|x|^2+12\left|v+\frac{x}{2}\right|^2-3|x|^2,
\end{align*}
which is \eqref{maj_par_H}. We then have
\begin{align*}
H\left(x,v\right)\geq\min\left(\frac{2}{3}\lambda,6\right)\left(|v|^2+|x+v|^2\right).
\end{align*}
Thus
\begin{align*}
r(x,\tilde{x},v,\tilde{v})^2&\leq\left(\left(1+\alpha\right)|x-\tilde{x}+v-\tilde{v}|+\alpha|v-\tilde{v}|\right)^2\\
&\leq2\left(1+\alpha\right)^2|x-\tilde{x}+v-\tilde{v}|^2+2\alpha^2|v-\tilde{v}|^2\\
&\leq4\left(\left(1+\alpha\right)^2+\alpha^2\right)\left(|x+v|^2+|v|^2+|\tilde{x}+\tilde{v}|^2+|\tilde{v}|^2\right).
\end{align*}
Therefore we obtain the final point.

\subsection{Proof of control of L1 and L2 Wasserstein distances}

We prove Lemma~\ref{rho_1_2}. Using the definition of $R_1$ and \eqref{eq:def_gamma_B}, and since $B\geq d\geq 1$ and $\gamma \leq \frac{1}{2}$, we have $R_1\geq 1$.
\begin{itemize}
\item First for the L1-Wasserstein distance
\begin{align*}
|x-x'|+|v-v'|\leq |v-v'+x-x'|+2|x-x'|\leq\max\left(\frac{2}{\alpha},1\right)r\left(\left(x,v\right),\left(x',v'\right)\right).
\end{align*}
If $r\left(\left(x,v\right),\left(x',v'\right)\right)\leq 1\leq R_1$
\begin{align*}
r\left(\left(x,v\right),\left(x',v'\right)\right)\leq \frac{f\left(r\right)}{f'_{-}\left(R_1\right)}\leq\frac{\rho\left(\left(x,v\right),\left(x',v'\right)\right)}{\phi\left(R_1\right)g\left(R_1\right)}.
\end{align*}
If $r\left(\left(x,v\right),\left(x',v'\right)\right)\geq 1$, we have shown \eqref{r_min_rho}
\begin{equation*}
r\left(\left(x,v\right),\left(x',v'\right)\right)\leq r^2\left(\left(x,v\right),\left(x',v'\right)\right)\leq4\frac{\left(1+\alpha\right)^2+\alpha^2}{\min\left(\frac{2}{3}\lambda,6\right)}\left(H\left(x,v\right)+H\left(x',v'\right)\right).
\end{equation*}
Thus
\begin{align*}
r\left(\left(x,v\right),\left(x',v'\right)\right)&\leq\frac{4}{\epsilon}\frac{\left(1+\alpha\right)^2+\alpha^2}{\min\left(\frac{2}{3}\lambda,6\right)}\left(\epsilon H\left(x,v\right)+\epsilon H\left(x',v'\right)\right)\\
&\leq\frac{4}{\epsilon}\frac{\left(1+\alpha\right)^2+\alpha^2}{\min\left(\frac{2}{3}\lambda,6\right)}\frac{\rho\left(\left(x,v\right),\left(x',v'\right)\right)}{f\left(r\right)}\\
&\leq\frac{4}{\epsilon}\frac{\left(1+\alpha\right)^2+\alpha^2}{\min\left(\frac{2}{3}\lambda,6\right)}\frac{\rho\left(\left(x,v\right),\left(x',v'\right)\right)}{f\left(1\right)}.
\end{align*}
Therefore
\begin{align*}
|x-x'|&+|v-v'|\leq\max\left(\frac{2}{\alpha},1\right)\max\left(\frac{4\left(\left(1+\alpha\right)^2+\alpha^2\right)}{\epsilon\min\left(\frac{2}{3}\lambda,6\right)f\left(1\right)},\frac{1}{\phi\left(R_1\right)g\left(R_1\right)}\right)\rho\left(\left(x,v\right),\left(x',v'\right)\right).
\end{align*}
\item Then for the L2-Wasserstein distance
\begin{align*}
|v-v'|^2=|v-v'+x-x'-\left(x-x'\right)|^2\leq2|v-v'+x-x'|^2+2|x-x'|^2.
\end{align*}
Hence
\begin{align*}
|x-x'|^2+|v-v'|^2\leq3\left(|v-v'+x-x'|^2+|x-x'|^2\right).
\end{align*}
But
\begin{align*}
r^2\left(\left(x,v\right),\left(x',v'\right)\right)&=\left(\alpha|x-x'|+|x-x'+v-v'|\right)^2\\
&\geq\alpha^2|x-x'|^2+|x-x'+v-v'|^2\\
&\geq \left(1+\alpha^2\right)\left(|x-x'|^2+|x-x'+v-v'|^2\right)\\
&\geq\frac{1+\alpha^2}{3}\left(|x-x'|^2+|v-v'|^2\right).
\end{align*}
If $r\left(\left(x,v\right),\left(x',v'\right)\right)\leq 1 \leq R_1$
\begin{align*}
r^2\left(\left(x,v\right),\left(x',v'\right)\right)\leq r\left(\left(x,v\right),\left(x',v'\right)\right)\leq \frac{f\left(r\right)}{f'_{-}\left(R_1\right)}\leq\frac{\rho\left(\left(x,v\right),\left(x',v'\right)\right)}{\phi\left(R_1\right)g\left(R_1\right)}.
\end{align*}
If $r\left(\left(x,v\right),\left(x',v'\right)\right)\geq 1$, we have shown \eqref{r_min_rho}
\begin{equation*}
r^2\left(\left(x,v\right),\left(x',v'\right)\right)\leq4\frac{\left(1+\alpha\right)^2+\alpha^2}{\min\left(\frac{2}{3}\lambda,6\right)}\left(H\left(x,v\right)+H\left(x',v'\right)\right).
\end{equation*}
Thus
\begin{align*}
r\left(\left(x,v\right),\left(x',v'\right)\right)&\leq\frac{4}{\epsilon}\frac{\left(1+\alpha\right)^2+\alpha^2}{\min\left(\frac{2}{3}\lambda,6\right)}\left(\epsilon H\left(x,v\right)+\epsilon H\left(x',v'\right)\right)\\
&\leq\frac{4}{\epsilon}\frac{\left(1+\alpha\right)^2+\alpha^2}{\min\left(\frac{2}{3}\lambda,6\right)}\frac{\rho\left(\left(x,v\right),\left(x',v'\right)\right)}{f\left(r\right)}\\
&\leq\frac{4}{\epsilon}\frac{\left(1+\alpha\right)^2+\alpha^2}{\min\left(\frac{2}{3}\lambda,6\right)}\frac{\rho\left(\left(x,v\right),\left(x',v'\right)\right)}{f\left(1\right)}.
\end{align*}
Therefore
\begin{align*}
|x-x'|^2&+|v-v'|^2\leq\frac{3}{1+\alpha^2}\max\left(\frac{4\left(\left(1+\alpha\right)^2+\alpha^2\right)}{\epsilon\min\left(\frac{2}{3}\lambda,6\right)f\left(1\right)},\frac{1}{\phi\left(R_1\right)g\left(R_1\right)}\right)\rho\left(\left(x,v\right),\left(x',v'\right)\right).
\end{align*}
\end{itemize}

%
%Subsection
%

\subsection{Proof of Lemma~\ref{lemma_dif_H}}\label{preuve_dif_H}

We have
\begin{align*}
H(x,v)-H(\tilde{x},\tilde{v})=&24\left(U(x)-U(\tilde{x})\right)+\left(6(1-\gamma)+\lambda\right)\left(|x|^2-|\tilde{x}|^2\right)+12\left(x\cdot v-\tilde{x}\cdot \tilde{v}\right)+12\left(|v|^2-|\tilde{v}|^2\right)\\
=&24\left(U(x)-U(\tilde{x})\right)+\left(6(1-\gamma)+\lambda-3\right)\left(|x|^2-|\tilde{x}|^2\right)+12\left(\left|v+\frac{x}{2}\right|^2-\left|\tilde{v}+\frac{\tilde{x}}{2}\right|^2\right).
\end{align*}
We first have
\begin{align*}
\left||x|^2-|\tilde{x}|^2\right| \leq\left|x-\tilde{x}\right|\left(|x|+|\tilde{x}|\right)\leq \frac{r(x,v,\tilde{x},\tilde{v})}{\alpha\sqrt{\lambda}}\left(\sqrt{H(x,v)}+\sqrt{H(\tilde{x},\tilde{v})}\right).
\end{align*}
Then
\begin{align*}
\left|\left|v+\frac{x}{2}\right|^2-\left|\tilde{v}+\frac{\tilde{x}}{2}\right|^2\right|\leq&\left|v+\frac{x}{2}-\tilde{v}-\frac{\tilde{x}}{2}\right|\left(\left|v+\frac{x}{2}\right|+\left|\tilde{v}+\frac{\tilde{x}}{2}\right|\right)\\
\leq&\frac{1}{\sqrt{12}}|v-\tilde{v}+\frac{1}{2}(x-\tilde{x})|\left(\sqrt{H(x,v)}+\sqrt{H(\tilde{x},\tilde{v})}\right)\\
\leq&\frac{1}{2\sqrt{3}}\max\left(1,\frac{1}{2\alpha}\right)r(x,v,\tilde{x},\tilde{v})\left(\sqrt{H(x,v)}+\sqrt{H(\tilde{x},\tilde{v})}\right).
\end{align*}
And finally
\begin{align*}
\left|U(x)-U(\tilde{x})\right|=&\left|\int_0^1\nabla U\left(\tilde{x}+t(x-\tilde{x})\right)\cdot (x-\tilde{x})dt\right|\\
\leq&\sup_{t\in[0,1]}|\nabla U\left(\tilde{x}+t(x-\tilde{x})\right)||x-\tilde{x}|\\
\leq&\left(\nabla U(0)+L_U(|x|+|\tilde{x}|)\right)|x-\tilde{x}|\\
\leq&\left(\nabla U(0)+\frac{L_U}{\sqrt{\lambda}}\left(\sqrt{H(x,v)}+\sqrt{H(\tilde{x},\tilde{v})}\right)\right)\frac{r(x,v,\tilde{x},\tilde{v})}{\alpha}.
\end{align*}
These three inequalities yield the desired result.

%
%section choix des parametres
%

\section{Proof of Lemma~\ref{Choix_constantes}}\label{choix_cstes}

We first rewrite the various conditions on the parameters. 

\begin{itemize}
\item Since for all $u\geq0$, $0<\phi\left(u\right)\leq1$, we have $0<\Phi\left(s\right)=\int_0^{s}\phi\left(u\right)du\leq s$, i.e $s/\Phi\left(s\right)\geq1$ .
Therefore
\begin{align*}
\inf_{r\in]0,R_1]}\frac{r\phi\left(r\right)}{\Phi\left(r\right)}\geq\inf_{r\in]0,R_1]}\phi\left(r\right)=\phi\left(R_1\right).
\end{align*}
It is thus sufficient for \eqref{cond_region_2_prop} that 
\begin{equation}\label{cond_c_2}
c+2\epsilon B\leq\frac{1}{2}\left(1-\frac{1}{\alpha}\left(L_U+L_W\right)\right)\phi\left(R_1\right).
\end{equation}
\item We have
\begin{align*}
\phi\left(r\right)\leq\exp\left(-\frac{L_U+L_W}{8\alpha}r^2\right).
\end{align*}
So
\begin{align*}
\Phi\left(r\right)\leq\int_0^{\infty}\exp\left(-\frac{L_U+L_W}{8\alpha}s^2\right)ds=\sqrt{\frac{2\pi\alpha}{L_U+L_W}}.
\end{align*}
Then
\begin{align*}
\int_0^{R_1}\frac{\Phi\left(r\right)}{\phi\left(r\right)}dr\leq\sqrt{\frac{2\pi\alpha}{L_U+L_W}}R_1\frac{1}{\phi\left(R_1\right)}.
\end{align*}
It is thus sufficient for \eqref{cond_g_demi_prop} that 
\begin{equation}\label{cond_c_3}
c+2\epsilon B\leq2\sqrt{\frac{L_U+L_W}{2\pi\alpha}}\frac{\phi\left(R_1\right)}{R_1}.
\end{equation}
\end{itemize}

At this point, we have now proven that under Assumption~\ref{HypU}, Assumption~\ref{HypU_lip} and Assumption~\ref{HypW}, for the parameters to satisfy Lemma~\ref{Choix_constantes} it is sufficient for them to satisfy
\begin{equation}\label{cond_fin_alpha}
\alpha>L_U+L_W,
\end{equation}
\begin{equation}\label{cond_c_1}
c \leq  \frac{\gamma }{6}\left(1-\frac{\frac{5}{6}\gamma }{2\epsilon B+\frac{5}{6}\gamma }\right),
\end{equation}
\begin{equation}\label{cond_c_2}
c+2\epsilon B\leq\frac{1}{2}\left(1-\frac{1}{\alpha}\left(L_U+L_W\right)\right)\phi\left(R_1\right),
\end{equation}
\begin{equation}\label{cond_c_3}
c+2\epsilon B\leq2\sqrt{\frac{L_U+L_W}{2\pi\alpha}}\frac{\phi\left(R_1\right)}{R_1},
\end{equation}
with, again
\begin{align*}
B&=24\left(A+ \left(\lambda-\gamma\right)\tilde{A}+d\right),\hspace{0.3cm} R_1=\sqrt{\left(1+\alpha\right)^2+\alpha^2}\sqrt{\frac{24 }{5\gamma 
\min\left(3,\frac{1}{3}\lambda\right)}B}.
\end{align*}
Let us show that there are positive parameters $\epsilon$, $\alpha$, $L_W$ and $c$ satisfying those conditions.

For inequality~\eqref{cond_fin_alpha} it is sufficient, as $L_W< \frac{\lambda}{8}$, to consider
\begin{equation*}
\alpha=L_U+\frac{\lambda}{4},
\end{equation*}
while inequality \eqref{cond_c_1} first invites us to consider $2\epsilon B$ of a comparable order to $c$
\begin{equation*}
2\epsilon B=\delta c.
\end{equation*}
We have thus switched parameter $\epsilon$ for $\delta$.
First we translate \eqref{cond_c_1} into our new parameter: 
\begin{align*}
c \leq \frac{\gamma }{6}\left(1-\frac{\frac{5}{6}\gamma }{2\epsilon B+\frac{5}{6}\gamma }\right)&\iff c \leq \frac{\gamma }{6}\frac{\delta c}{\delta c+\frac{5}{6}\gamma }\\
&\iff 1\leq\frac{\gamma }{6}\frac{\delta}{\delta c+\frac{5}{6}\gamma }\text{ (since $c\geq0$)}\\
&\iff c\leq\frac{\gamma }{6}\frac{\delta-5}{\delta}.
\end{align*}

The appearance of $\phi\left(R_1\right)$ in \eqref{cond_c_2} and \eqref{cond_c_3} suggests we should try to minimize it. Let us assume, for simplicity, that $\epsilon\leq 1$, which is equivalent to having $c\leq\frac{2B}{\delta}$.
We then have
\begin{align}
\phi\left(r\right)=&\exp\left(-\frac{1}{8}\left(\frac{1}{\alpha}\left(L_U+L_W\right)+\alpha+96\epsilon\max\left(\frac{1}{2\alpha},1\right)\right)r^2\right)\nonumber\\
\geq& \exp\left(-\frac{1}{8}\left(\frac{L_U+L_W}{\alpha}+\alpha+96\max\left(\frac{1}{2\alpha},1\right)\right)r^2\right)\text{ on }[0,R_1]\label{min_phi}.
\end{align}
Now, using \eqref{min_phi}, we have for \eqref{cond_c_2} and \eqref{cond_c_3} that it is sufficient that
\begin{equation*}
c\leq \frac{1}{2(\delta+1)}\left(1-\frac{1}{\alpha}\left(L_U+L_W\right)\right)\exp\left(-\frac{1}{8}\left(\frac{L_U+L_W}{\alpha}+\alpha+96\max\left(\frac{1}{2\alpha},1\right)\right)R_1^2\right),
\end{equation*}
and
\begin{equation*}
c\leq \frac{2}{\delta+1}\sqrt{\frac{L_U+L_W}{2\pi\alpha}}\exp\left(-\frac{1}{8}\left(\frac{L_U+L_W}{\alpha}+\alpha+96\max\left(\frac{1}{2\alpha},1\right)\right)R_1^2\right).
\end{equation*}

We could now optimize parameter $\delta$, but for the sake of conciseness,  we choose $\delta=6$.

 Recall $0\leq L_W<\frac{\lambda}{8}$. This way, both in \eqref{c_explicit} and \eqref{maj_L_W}, $c$ and $\mathcal{C}_1$ can be bounded independently of $L_W$. Hence why $L_W$, in \eqref{maj_L_W} and \eqref{cond_c_0}, can be chosen last, and every other quantities can be chosen independently of $L_W$.

%
%section choix des parametres
%

\section{$\nabla U$ locally Lipschitz-continuous}\label{Preuve_loc_lip}

In this section we replace Assumption~\ref{HypU_lip} with Assumption~\ref{HypU_loc_lip}. 
We assume, for $\nu^1_0$ and $\nu^2_0$ the initial conditions, 
\begin{equation}\label{init_cdt_loc_lip}
\forall i\in\{1,2\},\ \mathbb{E}_{\nu^i_0}\left(\left(\int_0^{H(X,V)}e^{a\sqrt{u}}du\right)^2\right)\leq (\mathcal{C}^0)^2
\end{equation}
We show how the proof can be modified to still obtain contraction. We make the following assumption
\begin{equation}\label{c_psi}
L_\psi\leq c_\psi(L_U, \lambda,\tilde{A},d,a) := L_U\frac{\gamma}{8}\min\left(\frac{a}{8},\epsilon_{\psi,*}\right)
\end{equation}
with $\tilde{B}$ given by \eqref{dyn_tilde_H_non_lin}-\eqref{Gronwall_exp_H} and $\gamma=\frac{\lambda}{2(\lambda+1)}$. Then, choose $\epsilon_{\psi,*}$ to be smaller than $\epsilon$, independently of $L_W$, using the assumption that $0\leq L_W\leq\frac{\lambda}{8}$. Like previously, we consider
\begin{equation*}
G_t=1+\epsilon  \tilde{H}\left(X_t,V_t\right)+\epsilon \tilde{H}(\tilde{X}_t,\tilde{V}_t).
\end{equation*}
Hence following the same method as previously we obtain
\begin{align}
K_t\leq&G_t\Big(c f\left(r_t\right)+\alpha f'\left(r_t\right)\frac{d|Z_t|}{dt}+ (L_U+L_W)f'\left(r_t\right) |Z_t|+L_Wf'\left(r_t\right) \mathbb{E}\left(|Z_t|\right)\label{K_t_loc_1}\\
&\hspace{1cm}+4 f''\left(r_t\right)rc^2\left(Z_t,W_t\right)\Big)+\frac{1}{2}\left(\epsilon\mathcal{C}_{f,1}+\mathcal{C}_{f,2}\right) r_tf'\left(r_t\right)rc\left(Z_t,W_t\right)^2\label{K_t_loc_2}\\
&+\epsilon\left(2\tilde{B}-\frac{\gamma }{8}\left(\tilde{H}\left(X_t,V_t\right) +\tilde{H}(\tilde{X}_t,\tilde{V}_t)\right)\right)f\left(r_t\right)\label{K_t_loc_3}\\
&+\left(\psi\left(X_t\right)+\psi(\tilde{X}_t)\right)|Z_t|f'\left(r_t\right)G_t-\frac{\gamma\epsilon}{8}\left(H\left(X_t,V_t\right)\exp\left(a\sqrt{H\left(X_t,V_t\right) }\right) +H(\tilde{X}_t,\tilde{V}_t)\exp\left(a\sqrt{H(\tilde{X}_t,\tilde{V}_t)}\right)\right)\label{K_t_loc_4}\\
&+\epsilon \frac{L_W}{\lambda}\left(6+8\lambda\right)\left(\exp\left(a\sqrt{H\left(X_t,V_t\right)}\right)\mathbb{E}H\left(X_t,V_t\right)+\exp\left(a\sqrt{H(\tilde{X}_t,\tilde{V}_t)}\right)\mathbb{E}H(\tilde{X}_t,\tilde{V}_t)\right)f\left(r_t\right).\label{K_t_loc_5}
\end{align}
We describe briefly how the terms will compensate each other before writing the calculations that are different.
\begin{itemize}
\item Like previously, lines \eqref{K_t_loc_1} and \eqref{K_t_loc_2} will 
be dealt with through the choice of function $f$, with the non linearity appearing at the end of \eqref{K_t_loc_1} giving us a remaining expectation (cf bullet \textbf{1} below),
\item line \eqref{K_t_loc_3} will intervene like before in the last region of space (where we use that for all $x$ in $\mathbb{R}$, $\tilde{H}\geq H$ to come back to calculations we've made in Section~\ref{region_3}, cf bullet \textbf{2} below) and in the first two region of space to compensate line \eqref{K_t_loc_4} (cf bullet \textbf{3} below),
\item and line \eqref{K_t_loc_5} will give us a remaining expectation (cf 
bullet \textbf{4} below).
\end{itemize}
Notice how we use the Lyapunov function to compensate $\psi$ appearing when considering $\nabla U$ only locally Lipschitz continuous. 
\begin{itemize}
\item[$\bullet$ \textbf{1}.] We can find a constant $\mathcal{C}_{1,e}$ such that for all $x,v,\tilde{x},\tilde{v}\in\mathbb{R}^d$,
\begin{align*}
|x-\tilde{x}|+|v-\tilde{v}|\leq\mathcal{C}_{1,e}\rho(x,v,\tilde{x},\tilde{v}),
\end{align*}
and thus
\begin{align*}
\mathbb{E}\left(\mathbb{E}\left(|Z_t|\right)G_tf'\left(r_t\right)\right)\leq\mathcal{C}_{1,e}\mathbb{E}\left(\rho_t\right)\mathbb{E}\left(G_t\right).
\end{align*}
\item[$\bullet$ \textbf{2}.] In the last region of space, we use the fact 
that
\begin{align}
K_t\mathds{1}_{R_3}\leq&\left(\left(c-\frac{\gamma}{8}\right)G_t+2\epsilon \tilde{B}+\frac{\gamma}{8}\right)f(r_t)\label{R_3_loc_1}
\end{align}
We deal with \eqref{R_3_loc_1} exactly like in Section~\ref{region_3}.
\item[$\bullet$ \textbf{3}.] To deal with the only locally Lipschitz continuous aspect, we use the upper bound on $\psi$ given in \eqref{c_psi}. In the first two regions of space we use $f'(r_t)|Z_t|\leq f'(r_t)r_t/\alpha\leq f(r_t)/\alpha$
\begin{align*}
G_t&\left(\psi\left(X_t\right)+\psi(\tilde{X}_t)\right)f'\left(r_t\right)|Z_t|\\
&-\epsilon\frac{\gamma }{8}\left(H\left(X_t,V_t\right)\exp\left(a\sqrt{H\left(X_t,V_t\right)}\right)+H(\tilde{X}_t,\tilde{V}_t)\exp\left(a\sqrt{H(\tilde{X}_t,\tilde{V}_t)}\right)\right)f\left(r_t\right)\\
&\leq\left(\psi\left(X_t\right)+\psi(\tilde{X}_t)\right)f\left(r_t\right)f'\left(r_t\right)|Z_t|\\
&\ \ +\frac{2\epsilon}{a\alpha}\left(\psi\left(X_t\right)+\psi(\tilde{X}_t)\right)\left( \sqrt{H\left(X_t,V_t\right)}\exp\left(a\sqrt{H\left(X_t,V_t\right)}\right)+\sqrt{H(\tilde{X}_t,\tilde{V}_t)}\exp\left(a\sqrt{H(\tilde{X}_t,\tilde{V}_t)}\right)\right)f\left(r_t\right)\\
&\ \ -\epsilon\frac{\gamma }{8}\left(H\left(X_t,V_t\right)\exp\left(a\sqrt{H\left(X_t,V_t\right)}\right)+H(\tilde{X}_t,\tilde{V}_t)\exp\left(a\sqrt{H(\tilde{X}_t,\tilde{V}_t)}\right)\right)f\left(r_t\right),
\end{align*}
where we used \eqref{control_tilde_H_2}. On one hand, since $\psi(x)\leq L_\psi\sqrt{H(x,v)}$, we have 
\begin{align*}
\psi(x)+\psi(\tilde{x})\leq L_\psi\sqrt{H(x,v)}+L_\psi\sqrt{H(\tilde{x},\tilde{v})}\leq L_\psi\left(\frac{H(x,v)+H(\tilde{x},\tilde{v})}{2}+1\right),
\end{align*}
and thus
\begin{align*}
\left(\psi(x)+\psi(\tilde{x})\right)&f'(r_t)|Z_t|\\
&\leq \frac{L_\psi}{2}\left(H(x,v)\exp\left(a\sqrt{H(x,v)}\right)+H(\tilde{x},\tilde{v})\exp\left(a\sqrt{H(\tilde{x},\tilde{v})}\right)\right)\frac{f(r_t)}{\alpha}+L_\psi f'(r_t)|Z_t|
\end{align*}
On the other hand
\begin{align*}
&\left(\psi(x)+\psi(\tilde{x})\right)\left(\sqrt{H(x,v)}\exp\left(a\sqrt{H(x,v)}\right)+\sqrt{H(\tilde{x},\tilde{v})}\exp\left(a\sqrt{H(\tilde{x},\tilde{v})}\right)\right)\\
&\ \ \leq L_\psi\left(\sqrt{H(x,v)}+\sqrt{H(\tilde{x},\tilde{v})}\right)\left(\sqrt{H(x,v)}\exp\left(a\sqrt{H(x,v)}\right)+\sqrt{H(\tilde{x},\tilde{v})}\exp\left(a\sqrt{H(\tilde{x},\tilde{v})}\right)\right)\\
&\ \ \leq 2L_\psi\left(H(x,v)\exp\left(a\sqrt{H(x,v)}\right)+H(\tilde{x},\tilde{v})\exp\left(a\sqrt{H(\tilde{x},\tilde{v})}\right)\right).
\end{align*}
This way, since $\frac{L_\psi}{2\alpha}\leq\epsilon\frac{\gamma}{16}$ and $\frac{4L_\psi\epsilon}{a\alpha}\leq\epsilon\frac{\gamma}{16}$ (recall $\alpha>L_U$), we get, since in the third region of space $f'\left(r_t\right)=0$,
\begin{align*}
G_t&\left(\psi\left(X_t\right)+\psi(\tilde{X}_t)\right)f'\left(r_t\right)|Z_t|\\
&-\epsilon\frac{\gamma }{8}\left(H\left(X_t,V_t\right)\exp\left(a\sqrt{H\left(X_t,V_t\right)}\right)+H(\tilde{X}_t,\tilde{V}_t)\exp\left(a\sqrt{H(\tilde{X}_t,\tilde{V}_t)}\right)\right)f\left(r_t\right)\leq L_\psi|Z_t|f'(r_t)G_t.
\end{align*}
The righthand side is dealt with through the choice of the concave function $f$ (we consider $L_U+ L_\psi$ instead of $L_U$).

\item[$\bullet$ \textbf{4}.] Likewise, we can bound
\begin{align*}
\mathbb{E}&\left(\epsilon\left(\exp\left(a\sqrt{H\left(X_t,V_t\right)}\right)\mathbb{E}H\left(X_t,V_t\right)+\exp\left(a\sqrt{H(\tilde{X}_t,\tilde{V}_t)}\right)\mathbb{E}H(\tilde{X}_t,\tilde{V}_t)\right)f\left(r_t\right)\right)\\
&\leq \mathcal{C}_{H,\tilde{H}}\mathbb{E}\left(\rho_t\right)+\mathcal{C}_{H,\tilde{H}}^0\mathbb{E}\left(\rho_t\right)e^{-\gamma  t},
\end{align*}
with $\mathcal{C}_{H,\tilde{H}}$ a constant independent of initial conditions and $\mathcal{C}_{H,\tilde{H}}^0$ another constant, possibly depending on initial conditions. Here, we used \eqref{gronwall_H} and \eqref{control_tilde_H}.
\end{itemize}
We can thus construct a function $f$ and a constant $c$, through the same 
calculations as before, such that there are $C$ and $C^0$ constants (resp. independent and dependent on initial conditions) such that
\begin{align*}
\forall t,\ \ e^{ct}\mathbb{E}\left(\rho_t\right)\leq\mathbb{E}\left(\rho_0\right)+\xi(1+\alpha)\mathbb{E}(G_t)e^{ct}+L_W C\int_0^t e^{cs}\mathbb{E}\left(\rho_s\right)ds+L_W C^0\int_0^t e^{\left(c-\gamma\lambda\right)s}\mathbb{E}\left(\rho_s\right)ds.
\end{align*}
Since $\mathbb{E}G_t$ is bounded uniformly in time, we may now conclude using Gronwall's lemma.

\subsection*{Acknowledgement}

This work has been funded by the project EFI ANR-17-CE40-0030.

\nocite{*}
\bibliographystyle{alpha}
\bibliography{biblio_2020_VF_longue}

\end{document}